\documentclass[a4paper, 11pt]{amsart}
\usepackage[utf8]{inputenc}
\usepackage[T1]{fontenc}
\usepackage{lmodern}
\usepackage[french,english]{babel}

\usepackage[final,unicode,psdextra]{hyperref}
\usepackage{amsmath}
\usepackage{amsthm}
\usepackage{amssymb,amsfonts}
\usepackage{mathtools}

\usepackage{framed}

\usepackage{tikz-cd}

\usepackage{rotating}

\usepackage{enumitem}

\DeclareMathSymbol{\shortminus}{\mathbin}{AMSa}{"39} 

\makeatletter
\@namedef{subjclassname@2020}{\textup{2020} Mathematics Subject Classification}
\makeatother

\tikzset{%
    symbol/.style={%
        draw=none,
        every to/.append style={%
            edge node={node [sloped, allow upside down, auto=false]{$#1$}}}
    }
}

\newcommand{\Ab}{\ensuremath{\mathrm{\textbf{Ab}}}}
\newcommand{\Set}{\ensuremath{\mathrm{\textbf{Set}}}}

\newcommand{\Ch}{\mathbf{Ch}_{\geq 0}} 

\newcommand{\Cat}{\mathbf{Cat}} 

\renewcommand{\C}{\ensuremath{\mathcal{C}}} 
\newcommand{\D}{\ensuremath{\mathcal{D}}} 
\newcommand{\W}{\ensuremath{\mathcal{W}}} 

\newcommand{\sD}{\ensuremath{\mathbb{D}}} 
\newcommand{\sH}{\ensuremath{\mathbb{H}}} 
\newcommand{\sS}{\ensuremath{\mathbb{S}}} 

\newcommand{\Hom}{\ensuremath{\mathrm{Hom}}}
\newcommand{\Ob}{\ensuremath{\mathrm{Ob}}}
\newcommand{\Psh}[1]{\ensuremath{\widehat{#1}}} 
\def\colim{\mathop{\mathrm{colim}}} 
\def\hocolim{\mathop{\mathrm{hocolim}}} 
\newcommand{\LL}{\ensuremath{\mathbb{L}}} 
\newcommand{\RR}{\ensuremath{\mathbb{R}}} 
\newcommand{\Fib}{\ensuremath{\mathrm{Fib}}} 
\newcommand{\Cof}{\ensuremath{\mathrm{Cof}}} 
\newcommand{\Ho}{\ensuremath{\mathrm{Ho}}} 

\newcommand{\nbd}{\nobreakdash-\hspace{0pt}} 

\newcommand{\Or}{\ensuremath{\mathcal{O}}} 

\theoremstyle{plain}
\newtheorem{theorem}{Theorem}[section]
\newtheorem{lemma}[theorem]{Lemma}

\newtheorem{proposition}[theorem]{Proposition}

\theoremstyle{definition}
\newtheorem{definition}[theorem]{Definition}
\newtheorem{remark}[theorem]{Remark}
\newtheorem{remarksub}{Remark}[theorem] 

\newtheorem{paragr}[theorem]{}

\theoremstyle{plain}
\newtheorem*{theorem*}{Theorem} 
\newtheorem{theoremintro}{Theorem} 

\theoremstyle{definition}
\newtheorem*{definition*}{Definition} 
\newtheorem*{nb*}{NB} 

\title{Homology of categories via polygraphic resolutions}
\author{L\'eonard Guetta}
\email{guetta@irif.fr}
\date{\today}

\address{Université de Paris, CNRS, IRIF, F-75006, Paris, France}

\begin{document}
\begin{abstract}
In this paper, we prove that the polygraphic homology of a small category, defined in terms of polygraphic resolutions in the category $\omega\Cat$ of strict $\omega$\nbd{}categories, is naturally isomorphic to the homology of its nerve, thereby extending a result of Lafont and Métayer. Along the way, we investigate homotopy colimits with respect to the Folk model structure and deduce a theorem which formally resembles Quillen's Theorem A.
\end{abstract}
\maketitle

\section*{Introduction}
In \cite{street1987algebra}, Street defines a nerve functor
\[N_{\omega} : \omega\Cat \to \Psh{\Delta}\]
from the category of strict $\omega$\nbd{}categories (that we shall simply call \emph{$\omega$\nbd{}categories}) to the category of simplicial sets. This functor can be used to transfer the homotopy theory of simplicial sets to $\omega$\nbd{}categories as it is done in the articles \cite{ara2014vers,ara2018theoreme,gagna2018strict,ara2019quillen,ara2020comparaison,ara2020theoreme,ara2016joint}. In particular, we can set the following definition:
\begin{definition*}
  Let $C$ be an $\omega$\nbd{}category and $k \in \mathbb{N}$. The $k$\nbd{}th homology group $H_k(C)$ of $C$ is the $k$\nbd{}th homology group of its nerve $N_{\omega}(C)$.
\end{definition*}

On the other hand, in \cite{metayer2003resolutions} Métayer defines other homological invariants for $\omega$\nbd{}categories, which we call here \emph{polygraphic homology groups}. The definition of these homology groups is based on the notion of \emph{free $\omega$\nbd{}category on a polygraph}, also known as \emph{free $\omega$\nbd{}category on a computad}, which are $\omega$\nobreakdash-categories obtained recursively from the empty $\omega$\nbd{}category by freely adjoining cells. From now on, we will simply say \emph{free $\omega$\nbd{}category}.

Métayer observed \cite[Definition 4.1 and Proposition 4.3]{metayer2003resolutions} that every $\omega$\nbd{}category $C$ admits a \emph{polygraphic resolution}, that is an arrow
\[
u : P \to C
\]
of $\omega\Cat$, such that $P$ is a free $\omega$\nbd{}category and $u$ satisfies some properties bearing formal similarities with trivial fibrations of topological spaces (or of simplicial sets). Moreover, every free $\omega$\nbd{}category $P$ can be ``linearized'' to a chain complex $\lambda(P)$ and Métayer proved \cite[Theorem 6.1]{metayer2003resolutions} that given $P \to C$ and $P' \to C$ two polygraphic resolutions of the same $\omega$\nbd{}category, the chain complexes $\lambda(P)$ and $\lambda(P')$ have the same homology groups. This leads to the following definition:
\begin{definition*}
Let $C$ be an $\omega$\nbd{}category and $k\in\mathbb{N}$. The $k$\nbd{}th polygraphic homology group $H_k^{\mathrm{pol}}(C)$ of $C$ is the $k$\nbd{}th homology group of $\lambda(P)$ for any polygraphic resolution $P \to C$. 
\end{definition*}
In this article, we prove the following theorem:
\begin{theoremintro}\label{maintheoremintro}
  Let $C$ be a small category. For every $k \in \mathbb{N}$, we have
  \[
  H_k(C) \simeq H_k^{\mathrm{pol}}(C).
  \]
\end{theoremintro}
For the statement of this theorem to make sense, we have to consider small categories as particular cases of $\omega$\nbd{}categories. Namely, a (small) category can be defined as an $\omega$\nbd{}category with only trivial cells in dimension greater than~$1$. Beware that this last property doesn't imply that the previous theorem is trivial: given  $P \to C$, a polygraphic resolution of a small category $C$, $P$ need \emph{not} have only trivial cells in dimension greater than $1$. 

The restriction of the previous theorem to the case of monoids seen as small
categories is exactly Corollary 3 of \cite[Section 3.4]{lafont2009polygraphic}.
As such, Theorem \ref{maintheoremintro} is only a small generalization of Lafont
and Métayer's result. However, the novelty lies in the new proof we give, which
is more conceptual than the one in \emph{loc.~cit.}

Note also that the actual result we shall obtain in this article (Theorem \ref{maintheorem}) is more precise than Theorem $1$. The first reason is that the homology of an $\omega$\nbd{}category (polygraphic or of the nerve) will be considered as a chain complex up to quasi-isomorphism and not only a sequence of abelian groups. The second and more important reason is that we will prove that the polygraphic homology and homology of the nerve of a small category are \emph{naturally} isomorphic and even explicitly construct the natural isomorphism. This last point was not addressed at all in \cite{lafont2009polygraphic}.

We shall now give a sketch of the proof of Theorem \ref{maintheoremintro}. It is slightly simpler than the proof of Theorem \ref{maintheorem} but has the same key ingredients. The simplification mainly results from avoiding questions of naturality.

The first step is to give a more abstract definition of the polygraphic homology. By a variation of the Dold--Kan equivalence (see for example \cite{bourn1990another}), the category $\omega\Cat(\Ab)$ of $\omega$\nbd{}categories internal to abelian groups is equivalent to the category $\Ch$ of chain complexes in non-negative degree. Thus, we have a forgetful functor
\[
\Ch \simeq \omega\Cat(\Ab) \to \omega\Cat,
\]
which has a left adjoint $\lambda : \omega\Cat \to \Ch$. Moreover, when $X$ is a
free $\omega$\nbd{}category, $\lambda(X)$ is exactly the linearization of $X$
considered in the definition of polygraphic homology by Métayer. Now,
$\omega\Cat$ admits a model structure, commonly referred to as the \emph{Folk
  model structure} \cite{lafont2010folk}, with the equivalences of
$\omega$\nbd{}categories (a generalization of the usual notion of equivalence of
categories) as weak equivalences and the free $\omega$\nbd{}categories as cofibrant objects \cite{metayer2008cofibrant}. As it turns out, if we equip $\omega\Cat$ with the Folk model structure and $\Ch$ with the projective model structure, then $\lambda$ is a left Quillen functor and hence admits a left derived functor
\[
\LL \lambda : \Ho(\omega\Cat) \to \Ho(\Ch).
\]
The polygraphic homology groups of an $\omega$\nbd{}category $X$ are exactly the homology groups of $\LL \lambda (X)$.\footnote{The description of polygraphic homology as a left derived functor has been around in the folklore for quite some time and I claim no originality for this result. For example, it will appear in \cite{polybook}.} We shall now simply write $H^{\mathrm{pol}}(X) := \LL \lambda(X)$.

Recall now that for every (small) category $C$, we have a canonical isomorphism
\[
\colim_{c \in C}C/c \simeq C,
\]
where $C/c$ is the slice over $c$. The realization that this colimit is a homotopy colimit with respect to the weak equivalences of the Folk model structure (Theorem \ref{thm:homcolfolk}) is at the origin of this paper. Since $\LL \lambda$ is the left derived functor of a left Quillen functor, it commutes with homotopy colimits. In particular, we have:
\[
H^{\mathrm{pol}}(C) \simeq \hocolim_{c \in C}H^{\mathrm{pol}} (C/c).
\]
Then, we can show that the polygraphic homology of a small category with a final
object is isomorphic to (the homology of) $\mathbb{Z}$ concentrated in
degree~$0$ (Lemma \ref{lemma:terminalobject} and Proposition \ref{prop:comparisoncontractible}). Hence, we have
\[
H^{\mathrm{pol}}(C) \simeq \hocolim_{c \in C}\mathbb{Z}.
\]
We conclude by remarking that the right-hand side of the previous equation is nothing but the homology of the nerve of $C$ (see for example \cite[Appendix II, Proposition 3.3]{gabriel1967calculus} or \cite[Section 1]{quillen1973higher}). Note that when $C$ is a monoid $M$, the category of functors $M \to \Ch$ is isomorphic to the category $\Ch(M)$ of chain complexes of left $\mathbb{Z}M$-modules and the $\colim_M$ functor can be identified with the functor
\[
\shortminus \bigotimes_{\mathbb{Z}M}\mathbb{Z} : \Ch(M) \to \Ch.
\]
Hence, in that case we also recover the definition of homology of a monoid in terms of $\mathrm{Tor}$ functors.

Let us end this introduction by mentioning that this paper is part of an ongoing program carried out by the author, which aims at understanding for which $\omega$\nbd{}categories $C$ the following holds:

\[
H^{\mathrm{pol}}_k(C) \simeq H_k(C) \text{ for every }k\geq 0.
\]

Theorem \ref{maintheoremintro} may lead us to think that all $\omega$\nbd{}categories satisfy this property but a counter-example discovered by Ara and Maltsiniotis shows that it is not the case: Let $C$ be the commutative monoid $(\mathbb{N},+)$ considered as a $2$\nobreakdash -category with only one object and no non-trivial cells in dimension $1$. This $2$-category is free (as an $\omega$\nbd{}category) and a quick computation shows that
\[
H^{\mathrm{pol}}_k(C) \simeq
\begin{cases}
  \mathbb{Z} \text{ for } k=0, 2\\
  0 \text{ otherwise. }
\end{cases}
\]
But, as proved in \cite[Theorem 4.9 and Example 4.10]{ara2019quillen}, the nerve of $C$ is a $K(\mathbb{Z},2)$ which has non-trivial homology groups in arbitrarily high even dimension.  

\section{Generalities on homotopy colimits}
\emph{The goal of this section is to provide a short summary on homotopy
  colimits. The reader familiar with the subject may skip it and refer to it if needed.}
\begin{paragr}
  A \emph{localizer} is a pair $(\C,\W)$ where $\C$ is a category and $\W$ is a class of arrows of $\C$, which we usually refer to as the weak equivalences.
We denote by $\Ho(\C)$, the localization of $\C$ with respect to $\W$ and by
  \[
  \gamma : \C \to \Ho(\C)
  \]
  the localization functor \cite[1.1]{gabriel1967calculus}. Recall the universal property of the localization: for every category $\D$, the functor induced by pre-composition
  \[
  \gamma^* : \underline{\Hom}(\Ho(\C),\D) \to \underline{\Hom}(\C,\D)
  \]
  is fully faithful and its essential image consists of the functors $F : \C
  \to \D$ that send the morphisms of $\W$ to isomorphisms of $\D$.
  
  We shall always consider that $\C$ and $\Ho(\C)$ have the same class of objects and implicitly use the equality
  \[
  \gamma(X)=X
  \]
  for every object $X$ of $\C$.
\end{paragr}
\begin{paragr}
Let $(\C,\W)$ and $(\C',\W')$ be two localizers and  $F : \C \to \C'$ be a functor. If $F$ \emph{preserves weak equivalences}, i.e.\ $F(\W) \subseteq \W'$, then the universal property of the localization implies that there is a canonical functor
  \[
  \overline{F} : \Ho(\C) \to \Ho(\C')
  \]
  such that the square
    \[
  \begin{tikzcd}
    \C \ar[r,"F"] \ar[d,"\gamma"] & \C' \ar[d,"\gamma'"]\\
    \Ho(\C) \ar[r,"\overline{F}"] & \Ho(\C').
  \end{tikzcd}
  \]
  is commutative.
\end{paragr}
\begin{remarksub}\label{remark:homotopicalfunctor}
  Since we always consider that localization functors are the identity on
  objects, we have the
  equality
  \[
   \overline{F}(X)=F(X)
 \]
 for every object $X$ of $\C$.
\end{remarksub}
\begin{paragr}\label{paragr:defleftderived}
  Let $(\C,\W)$ and $(\C',\W')$ be two localizers. A functor $F : \C \to \C'$ is \emph{totally left derivable} when there exists a functor
  \[
  \LL F : \Ho(\C) \to \Ho(\C')
  \]
  and a natural transformation 
  \[
  \alpha : \LL F \circ \gamma \Rightarrow \gamma'\circ F
  \]
  that makes $\LL F$ the \emph{right} Kan extension of $\gamma' \circ F$ along $\gamma$:
  \[
  \begin{tikzcd}
    \C \ar[r,"F"] \ar[d,"\gamma"] & \C' \ar[d,"\gamma'"]\\
    \Ho(\C) \ar[r,"\LL F"'] & \Ho(\C').
    \arrow[from=2-1, to=1-2,"\alpha",Rightarrow]
  \end{tikzcd}
  \]
 When this right Kan extension is \emph{absolute}, we say that $F$ is \emph{absolutely totally left derivable}.

  Note that when a functor $F$ is totally left derivable, the pair $(\LL F,\alpha)$ is unique up to a unique natural isomorphism and thus we shall refer to $\LL F$ as \emph{the} total left derived functor of $F$.

  The notion of (absolute) total right derivable functor is defined dually and
  the notation $\RR F$ is used.
\end{paragr}
\begin{remarksub}\label{rem:homotopicalisleftder}
    If $F : \C \to \C'$ preserves weak equivalences, then it follows from the
    universal property of the localization that $F$ is absolutely totally left
    and right derivable and $\LL F \simeq \RR F \simeq \overline{F}$.
  \end{remarksub}
\begin{paragr}\label{paragr:diagfun}
  Let $(\C,\W)$ be a localizer and $A$ be a small category. We denote by $\C^A$ the category of functors from $A$ to $\C$ and natural transformations between them. An arrow $\alpha : d \to d'$ of $\C^A$ is a \emph{pointwise weak equivalence} when $\alpha_a : d(a) \to d'(a)$ belongs to $\W$ for every $a \in A$. We denote by $\W_A$ the class of pointwise weak equivalences. This defines a localizer $(\C^A,\W_A)$.

  Let 
  \[
    k : \C \to \C^A
  \]
  be the diagonal functor, i.e.\ for an object $X$ of $\C$, $k(X) : A \to \C$ is the constant functor with value $X$.  This functor preserves weak equivalences, whence a functor
  \[
  \overline{k} : \Ho(\C) \to \Ho(\C^A).
  \]
\end{paragr}
\begin{definition}
  A localizer $(\C,\W)$ has \emph{homotopy colimits} when for every small category $A$, the functor
  \[
  \overline{k} : \Ho(\C) \to \Ho(\C^A)
  \]
  has a left adjoint.
\end{definition}
\begin{paragr}\label{paragr:hocolim}
  When a localizer $(\C,\W)$ has homotopy colimits, we denote by \[\hocolim_A : \Ho(\C) \to \Ho(\C^A)\] the left adjoint of $\overline{k} : \Ho(\C) \to \Ho(\C^A)$. For an object $d$ of $\C^A$, the object
  \[
  \hocolim_A (d)
  \]
  of $\Ho(\C)$ is the \emph{homotopy colimit} of $d$. For consistency, we also use the notation
  \[
  \hocolim_{a \in A}d(a).
  \]

  Note that when $\C$ has colimits and $(\C,\W)$ has homotopy colimits, it
  follows from Remark \ref{rem:homotopicalisleftder} and the dual of
  \cite[Theorem 3.4]{gonzalez2012derivability} that $\hocolim_A$ is the total left derived functor of $\colim_A$. In particular, for every functor $d : A \to \C$, there is a canonical arrow of $\Ho(\C)$
  \[
  \hocolim_A (d) \to \colim_A (d).
  \]
\end{paragr}

\begin{paragr}\label{paragr:canmaphocolim}
  Let $(\C,\W)$ and $(\C',\W')$ be two localizers and $F : \C \to \C'$ be a functor that preserves weak equivalences. For every small category $A$, the functor induced by post-composition, which we abusively denote by
  \[
  F: \C^A \to \C'^A,
  \]
  again preserves weak equivalences and we have a commutative square
  \[
  \begin{tikzcd}
    \Ho(\C^A) \ar[r,"\overline{F}"] & \Ho(\C'^A) \\
    \Ho(\C) \ar[u,"\overline{k}"] \ar[r,"\overline{F}"] & \Ho(\C') \ar[u,"\overline{k}"].
  \end{tikzcd}
  \]
  Suppose now that $(\C,\W)$ and $(\C',\W')$ have homotopy colimits. Using the unit and co-unit of the adjunctions $\hocolim_A \dashv \overline{k}$,
  \[
  \begin{tikzcd}[column sep=large]
    &\Ho(\C^A) \ar[r,"\overline{F}"] & \Ho(\C'^A)\ar[r,"\hocolim_A"]& \Ho(\C')\\
    \Ho(\C^A)\ar[r,"\hocolim_A"'] \ar[ru,"\mathrm{id}",""{name=U,below},bend left]&\Ho(\C) \ar[u,"\overline{k}"] \ar[r,"\overline{F}"] & \Ho(\C') \ar[u,"\overline{k}"]\ar[ru,"\mathrm{id}"',bend right,""{name=C,above}]&,
    \ar[Rightarrow,from=U,to=2-2,"\eta"]\ar[Rightarrow,from=1-3,to=C,"\epsilon"]
  \end{tikzcd}
  \]
  we obtain a natural transformation
  \[
      \hocolim_A \circ \, \overline{F} \Rightarrow \overline{F} \circ \hocolim_A.
  \]
  Hence, for every object $d : A \to \C$ of $\C^A$, a canonical map
  \[
  \hocolim_A(F(d)) \to F(\hocolim_A(d)).
  \]
\end{paragr}
\section{Homotopy colimits in combinatorial model categories}
\emph{In this section, we recall a few useful results on
  homotopy colimits in the context of model categories. As for the previous
  section, it might be skipped at first reading and referred to when needed. We suppose however
  that the reader is familiar with the basics of model category theory. Recall that a combinatorial model category is a cofibrantly generated model category such that the underlying category is locally presentable.}
\begin{paragr}\label{paragr:quilfun}
 Let $(\C,\W,\Cof,\Fib)$ and $(\C',\W',\Cof',\Fib')$ be two model categories and
 $F : \C \to \C'$ be a functor.  Recall that if $F$ is a left Quillen functor (i.e.\ the left adjoint in a Quillen adjunction), then $F$ is absolutely totally left derivable and for every cofibrant object $X$ of $\C$, the canonical arrow
  \[
  \LL F(X) \to F(X)
  \]
  is an isomorphism of $\Ho(\C')$.
  
\end{paragr}
\begin{remarksub} Note that the definition of totally left derivable functor
  only depends on the weak equivalences. In particular, if there are several model structures for which a functor $F$ is left Quillen, then for any cofibrant object $X$ of \emph{any} such model structure, the canonical map
  \[
  \LL F(X) \to F(X)
  \]
  is an isomorphism.
  \end{remarksub}

  \begin{proposition}\label{prop:modprs}
    Let $(\C,\W,\Cof,\Fib)$ be a combinatorial model category. For every small category $A$, there exists:
    \begin{enumerate}
    \item A model structure on $\C^A$, called the \emph{projective model
        structure}, with the pointwise weak equivalences and the pointwise fibrations as weak equivalences and fibrations respectively.
    \item A model structure on $\C^A$, called the \emph{injective model
        structure}, with the pointwise weak equivalences and the pointwise cofibrations as weak equivalences and cofibrations respectively.
      \end{enumerate}
  \end{proposition}
  \begin{proof}
    See \cite[Proposition A.2.8.2]{lurie2009higher}
  \end{proof}
  \begin{proposition}
    Let $(\C,\W,\Cof,\Fib)$ be a combinatorial model category and $A$ a small category. The adjunction
    \[
    \begin{tikzcd}
      \colim_A : \C^A \ar[r,shift left]& \ar[l,shift left] \C : k
      \end{tikzcd}
    \]
    is a Quillen adjunction with respect to the projective model structure on $\C^A$.
    \end{proposition}
  \begin{proof}
By definition of the projective model structure, $k$ preserve weak equivalences and fibrations.
  \end{proof}
  \begin{paragr}\label{paragr:hocolimms}
    We deduce from the previous proposition that if a localizer $(\C,\W)$ can be extended to a combinatorial model category $(\C,\W,\Cof,\Fib)$, then it has homotopy colimits and
    \[
    \hocolim_A \simeq \LL \colim_A.
    \]
    Since $\colim_A$ is left Quillen with respect to the projective model structure, it is particularly interesting to detect the cofibrant objects of this model structure in order to compute homotopy colimits. This is the goal of the paragraph below and the lemma that follows.
    \end{paragr}
    \begin{paragr}\label{paragr:cofprojms}
    Let $\C$ be a category with coproducts and $A$ a small category. For every
    object $X$ of $\C$ and every object $a$ of $A$, we define $X\otimes a$ as the functor
  \[
  \begin{aligned}
    X \otimes a : A &\to \C \\
    b &\mapsto \coprod_{\Hom_A(a,b)}X.
    \end{aligned}
  \]
  For every object $a$ of $A$, this gives rise to a functor
   \[
  \begin{aligned}
  \shortminus \otimes a : \C &\to \C^A\\
  X &\mapsto X \otimes a.
  \end{aligned}
  \]
  \end{paragr}
  \begin{lemma}\label{lemma:cofprojms}
    Let $\C=(\C,\W,\Cof,\Fib)$ be a combinatorial model category and $A$ a small category. For every object $a$ of $A$ and every cofibration $f : X \to Y$ of $\C$, the arrow
    \[
    f \otimes a : X \otimes a \to Y \otimes a
    \]
    is a cofibration of the projective model structure on $\C^A$.
  \end{lemma}
  \begin{proof}
  We leave it to the reader to check that the functor $\shortminus \otimes a$ is left adjoint to
  \[
  \begin{aligned}
    \mathrm{ev}_a : \C^A &\to \C\\
    F &\mapsto F(a).
    \end{aligned}
  \]
 Let $\alpha$ be a fibration of the projective model structure on $\C^A$. By definition, $f$ has the left lifting property with respect to $\mathrm{ev}_a(\alpha)$. Hence, by adjunction, $f\otimes a$ has the left lifting property with respect to $\alpha$, which is what we needed to prove.
  \end{proof}
  \begin{paragr}
    Let $(\C,\W,\Cof,\Fib)$ and $(\C',\W',\Cof',\Fib')$ be two combinatorial
    model categories and let $F : \C \to \C'$ be a left Quillen functor. For every small category $A$, the functor induced by post-composition
    \[
    F : \C^A \to \C'^A
    \]
    is left Quillen both with respect to the projective model structure \emph{and} the injective model structure. In particular, all arrows of the commutative square
    \[
    \begin{tikzcd}
          \C^A \ar[r,"F"] & \C'^A \\
          \C \ar[u,"k"] \ar[r,"F"] & \C' \ar[u,"k"]
      \end{tikzcd}
    \]
    are left Quillen functors when $\C^A$ and $\C'^A$ are equipped with the injective model structure. Hence, by composition of left Quillen functors, we obtain a commutative square up to a canonical isomorphism 
    \[
        \begin{tikzcd}
          \Ho(\C^A) \ar[r,"\LL F"] & \Ho(\C'^A) \\
          \Ho(\C) \ar[u,"\overline{k}"] \ar[r,"\LL F"] & \Ho(\C') \ar[u,"\overline{k}"].
            \ar[phantom,from=1-1, to=2-2,"\simeq",description]
      \end{tikzcd}
        \]
        In a similar fashion as in Paragraph \ref{paragr:canmaphocolim}, we obtain a natural transformation
        \[
        \hocolim_A \circ \,\LL F \Rightarrow \LL F \circ \hocolim_A.
        \]
        The next proposition tells us that left Quillen functors are ``homotopy cocontinuous''. 
  \end{paragr}
  
  \begin{proposition}\label{prop:leftquilcocontinuous}
    Let $(\C,\W,\Cof,\Fib)$ and $(\C',\W',\Cof',\Fib')$ be two combinatorial
    model categories, $F : \C \to \C'$ be a left Quillen functor and $A$ be a small category. For every object $d : A \to \C$ of $\C^A$, the canonical map
    \[
    \hocolim_A (\LL F(d)) \to \LL F (\hocolim_A(d))
    \]
    is an isomorphism of $\Ho(\C')$.
  \end{proposition}
  \begin{proof}
    We use here the projective model structures on $\C^A$ and $\C'^A$. Since every object $d$ of $\C^A$ is isomorphic in $\Ho(\C^A)$ to a cofibrant one, it suffices to show that the map
    \[
    \hocolim_A(\LL F(d)) \to \LL F(\hocolim_A(d))
    \]
    is an isomorphism when $d$ is cofibrant. But in this case, the previous map can be identified with the image of the map
    \[
    \colim_A(F(d))\to F(\colim_A(d))
    \]
    in the localization of $\C'^A$. Since $F$ is a left adjoint, this map is indeed an isomorphism.
  \end{proof}
  \section{\texorpdfstring{$\omega$}{\textomega}-categories}
  \begin{paragr}
    We denote by $\omega \Cat$ the category of (small) \emph{strict} $\omega$\nbd{}categories and \emph{strict} $\omega$\nbd{}functors. Since we shall never deal with non-strict $\omega$\nbd{}categories, we omit the word ``strict'' and simply say $\omega$\nbd{}categories and $\omega$\nbd{}functors. For an $\omega$\nbd{}category $C$, we denote by $C_n$ the set of $n$\nbd{}cells of $C$. For $x \in C_n$, we refer to the integer $n$ as the \emph{dimension of $x$}.

    For an $n$\nbd{}cell $x$ with $n>0$, $s(x)$ and $t(x)$ are respectively the source and target of $x$ (which are $(n \shortminus 1)$-cells). More generally, for every $k<n$, $s_k(x)$ and $t_k(x)$ are respectively the $k$\nbd{}dimensional source of $x$ and the $k$\nbd{}dimensional target of $x$, which are obtained by iteration. We say that two $n$\nbd{}cells $x$ and $y$ are \emph{parallel} if either $n=0$ or $n>0$ and
    \[
    s(x)=s(y) \text{ and }t(x)=t(y).
    \]
    Let $k<n \in \mathbb{N}$. Two $n$\nbd{}cells $x$ and $y$ are \emph{$k$\nbd{}composable} if $s_k(x)=t_k(y)$, in which case we denote their $k$\nbd{}composition by $x\ast_k y$.
    
    For an $n$\nbd{}cell $x$, $1_x$ is the unit on $x$, which is an $(n+1)$\nbd{}cell. More generally, for every $k>n$ we denote by $1^{(k)}_x$ the $k$\nbd{}dimensional unit on $n$ obtained by iteration. (Note that the superscript denotes the dimension and not the number of iterations.) For consistency, we also define $1^{(n)}_x:=x$ for every $n$\nbd{}cell $x$. An $n$\nbd{}cell $x$ is \emph{degenerate} if there exists a cell $y$ of dimension strictly lower than $n$ and such that $x=1^{(n)}_y$.

    We extend the notion of $k$\nbd{}composition for cells of different dimension in the following way. Let $x$ be an $n$\nbd{}cell, $y$ be an $m$\nbd{}cell and $k < \min\{m,n\}$. The cells $x$ and $y$ are \emph{$k$\nbd{}composable} if $s_k(x)=t_k(y)$, in which case we define $x\ast_k y$ as the cell of dimension $l := \max\{m,n\}$
    \[
    x\ast_ky:= 1^{(l)}_x\ast_k 1^{(l)}_y.
    \]

    Finally, we use the convention that for $n<m$ the operation $\ast_n$ has priority over $\ast_m$, which means that
    \[
    x\ast_n y \ast_m z = (x \ast_n y) \ast_m z \text{ and } x \ast_m y \ast_n z = x \ast_m (y \ast_n z)
    \]
    whenever these equations make sense.
  \end{paragr}
  \begin{paragr}
    Let $n \in \mathbb{N}$. An \emph{$n$\nbd{}category} is an $\omega$\nbd{}category such that all cells of dimension strictly greater than $n$ are degenerate. An \emph{$n$\nbd{}functor} is an $\omega$\nbd{}functor between $n$\nbd{}categories. We denote by $n\Cat$ the category of $n$\nbd{}categories and $n$\nbd{}functors.
    For an $\omega$\nbd{}category $C$, we denote by $C_{\leq n}$ the $n$\nbd{}category obtained from $C$ by removing all non-degenerate cells of dimension strictly greater than $n$. The functor
    \[
    \begin{aligned}
      \omega\Cat &\to n\Cat\\
      C &\mapsto C_{\leq n}
      \end{aligned}
    \]
    is right adjoint to the canonical inclusion functor
    \[
    n\Cat \hookrightarrow \omega\Cat.
    \]
    For an $\omega$\nbd{}category $C$, we define $C_{\leq -1}$ to be the empty $\omega$\nbd{}category (which is a $(\shortminus 1)$-category).
  \end{paragr}
  \begin{remark}\label{rem:1cat}
    When $n=1$, we have a canonical functor
    \[
    1\Cat \to \Cat
    \]
    from the category of $1$-categories to the category of small categories, which simply forgets the $k$\nbd{}cells for $k>1$. Since this functor is an isomorphism of categories, we usually identify the categories $\Cat$ and $1\Cat$ and consider that the terms ``$1$-category'' and ``(small) category'' are synonyms. 
  \end{remark}
  \begin{paragr}\label{paragr:defglobe}
    Let $n \in \mathbb{N}$. The functor
    \[
    \begin{aligned}
      \omega\Cat &\to \Set\\
      C &\mapsto C_n
      \end{aligned}
    \]
    is representable and we define the $n$\nbd{}globe $\sD_n$ as the $\omega$\nbd{}category representing this functor. ($\sD_n$ is in fact an $n$\nbd{}category.) 
    For an $n$\nbd{}cell $x$ of an $\omega$\nbd{}category $C$, we denote by
    \[
    \langle x \rangle : \sD_n \to C
    \]
    the associated morphism of $\omega\Cat$. Here are some pictures of $\sD_n$ in low dimension:
     \[
  \sD_0= \begin{tikzcd}\bullet\end{tikzcd},
  \]
  \[
  \sD_1 = \begin{tikzcd} \bullet \ar[r] &\bullet \end{tikzcd},
  \]
  \[
  \sD_2 = \begin{tikzcd}
    \bullet \ar[r,bend left=50,""{name = U,below}] \ar[r,bend right=50,""{name=D}]&\bullet \ar[Rightarrow, from=U,to=D]
  \end{tikzcd},
  \]
    \[
    \sD_3 = \begin{tikzcd}
         \bullet \ar[r,bend left=50,""{name = U,below,near start},""{name = V,below,near end}] \ar[r,bend right=50,""{name=D,near start},""{name = E,near end}]&\bullet \ar[Rightarrow, from=U,to=D, bend right,""{name= L,above}]\ar[Rightarrow, from=V,to=E, bend left,""{name= R,above}]
    \arrow[phantom,"\Rrightarrow",from=L,to=R]
  \end{tikzcd}.
    \]
    For an $\omega$\nbd{}category $C$, let us denote by $\mathrm{Par}_n(C)$ the set of pair of parallel $n$\nbd{}cells of $C$. The functor
    \[
    \begin{aligned}
      \omega\Cat &\to \Set\\
      C &\mapsto \mathrm{Par}_n(C)
      \end{aligned}
    \]
    is representable and we define the $n$\nbd{}sphere $\sS_n$ as the $\omega$\nbd{}category representing this functor. ($\sS_n$ is in fact an $n$\nbd{}category.) If $x$ and $y$ are parallel $n$\nbd{}cells of $C$, we denote by
    \[
    \langle x,y\rangle : \sS_n \to C
    \]
    the associated morphism of $\omega\Cat$. Here are some pictures of $\sS_n$ in low dimension:
       \[
  \sS_0= \begin{tikzcd}\bullet & \bullet \end{tikzcd},
  \]
  \[
  \sS_1 = \begin{tikzcd}     \bullet \ar[r,bend left=50,""{name = U,below}] \ar[r,bend right=50,""{name=D}]&\bullet  \end{tikzcd},
  \]
  \[
  \sS_2 = \begin{tikzcd}
         \bullet \ar[r,bend left=50,""{name = U,below,near start},""{name = V,below,near end}] \ar[r,bend right=50,""{name=D,near start},""{name = E,near end}]&\bullet \ar[Rightarrow, from=U,to=D, bend right,""{name= L,above}]\ar[Rightarrow, from=V,to=E, bend left,""{name= R,above}]
    \end{tikzcd}.
  \]
  Suppose now that $n>0$, and let $C$ be an $\omega$\nbd{}category. The canonical morphism
  \[
  \begin{aligned}
    C_n &\to \mathrm{Par}_{n-1}(C)\\
    x &\mapsto (s(x),t(x)),
    \end{aligned}
  \]
 is natural in $C$. Hence, a canonical morphism
  \[
  i_n : \sS_{n-1}\to \sD_n.
  \]
  We also define $\sS_{- 1}$ to be the empty $\omega$\nbd{}category (which is the initial object of $\omega\Cat$), and $i_0$ to be the unique morphism
  \[
  i_0 : \sS_{- 1} \to \sD_0.
  \]
  \end{paragr}
  \begin{paragr}
    A \emph{basis} of an $\omega$\nbd{}category $C$ is a graded set of cells of $C$
    \[
    \Sigma = (\Sigma_n \subseteq C_n)_{n\in \mathbb{N}}
    \]
   such that for every $n\in\mathbb{N}$, the commutative square
        \[
    \begin{tikzcd}[column sep=huge,row sep=huge]
      \displaystyle\coprod_{x \in \Sigma_n}\sS_{n-1} \ar[d,"\displaystyle\coprod_{x \in \Sigma_n} i_n"'] \ar[r,"{\langle s(x) , t(x) \rangle}"] & C_{\leq n-1}\ar[d]\\
      \displaystyle\coprod_{x \in \Sigma_n}\sD_n \ar[r,"\langle x \rangle"] & C_{\leq n},
      \end{tikzcd}
    \]
    where the anonymous arrow is the canonical inclusion, is a pushout square. That is, $C_{\leq n}$ is obtained from $C_{\leq n-1}$ by freely adjoining the $n$\nbd{}cells that belongs to $\Sigma_n$. An $\omega$\nbd{}category is \emph{free} when it has a basis.
    Note that a free $\omega$\nbd{}category has a \emph{unique} basis (see \cite[Section 4, Proposition 8.3]{makkai2005word}). This allows us to speak of \emph{the} basis of a free $\omega$\nbd{}category. Cells that belong to $\Sigma_n$ are referred to as \emph{generating $n$\nbd{}cells} of $C$.
  \end{paragr}
   Recall from \cite{guetta2020polygraphs} the following definition:
    \begin{definition}\label{def:condfun}
      An $\omega$\nbd{}functor $f : C \to D$ is a \emph{discrete Conduché
        $\omega$\nbd{}functor} if for every  $n\geq 0$, for every
      $n$\nbd{}cell $x$ of $C$ and for every pair $(y_1,y_2)$ of
      $k$\nbd{}composable $n$\nbd{}cells of $D$ with $k<n$ such that
      \[
      f(x)=y_1\ast_ky_2,
      \]
      there exists a \emph{unique} pair $(x_1,x_2)$ of $k$\nbd{}composable $n$\nbd{}cells of $C$ such that
      \begin{enumerate}
      \item $x=x_1\ast_kx_2$
        \item $f(x_1)=y_1$ and $f(x_2)=y_2$.
      \end{enumerate}
      \end{definition}
      \begin{lemma}\label{lemma:condfunpullback}
        Let
        \[
        \begin{tikzcd}
          C' \ar[r,"u"] \ar[d,"f'"] & C \ar[d,"f"] \\
          D' \ar[r,"v"] & D
          \end{tikzcd}
        \]
        be a pullback square in $\omega\Cat$. If $f$ is a discrete Conduché $\omega$\nbd{}functor, then so is $f'$.
      \end{lemma}
      \begin{proof}
        Left to the reader. See \cite[remark 4.5]{guetta2020polygraphs}.
      \end{proof}
      \begin{proposition}\label{prop:condfunfree}
        Let $f : C \to D$ be a discrete Conduché $\omega$\nbd{}functor. If $D$ is a free $\omega$\nbd{}category, then so is $C$.
        
        More precisely, if we denote by $\Sigma^D_n$ the set of generating $n$\nbd{}cells of $D$, then the set of generating $n$\nbd{}cells of $C$ is
        \[
        \Sigma^C_n=\{x \in C_n | f(x) \in \Sigma^D_n\}.
        \]
      \end{proposition}
      \begin{proof}
      This is Theorem 5.12(1) from \cite{guetta2020polygraphs}.
      \end{proof}
   
  \begin{paragr}
    Let $C$ be an $\omega$\nbd{}category. We define the equivalence relation $\sim_{\omega}$ on the set $C_n$ by co-induction on $n \in \mathbb{N}$. Let $x, y \in C_n$, then $x \sim_{\omega} y $ when:
    \begin{itemize}
    \item[-] $x$ and $y$ are parallel,
    \item[-] there exist $r, s \in C_{n+1}$ such that $r : x \to y$, $s : y \to x$,
      \[
      r\ast_{n}s \sim_{\omega} 1_y
      \]
      and
      \[
      s\ast_nr \sim_{\omega} 1_x.
      \]

    \end{itemize}
    For details on this definition and the proof that it is an equivalence relation, see \cite[section 4.2]{lafont2010folk}.
    \end{paragr}
  \begin{paragr}\label{paragr:defeqomegacat}
    An $\omega$\nbd{}functor $f : C \to D$ is an \emph{equivalence of $\omega$\nbd{}categories} when:
    \begin{itemize}
      \item[-] for every $y \in D_0$, there exists a $x \in C_0$ such that
      \[f(x)\sim_{\omega}y,\]
    \item[-] for every $x,y \in C_n$ that are \emph{parallel} and every $\beta \in D_{n+1}$ such that \[\beta : f(x) \to f(y),\] there exists $\alpha \in C_{n+1}$ such that
      \[\alpha : x \to y
      \]
      and
      \[f(\alpha)\sim_{\omega}\beta.\]
      \end{itemize}
  \end{paragr}
  \begin{theorem}\label{thm:folkms}
    There exists a combinatorial model structure on $\omega\Cat$ whose weak
    equivalences are the equivalences of $\omega$\nbd{}categories, and such that
    the set $\{i_n : \sS_{n-1} \to \sD_n \vert n \in \mathbb{N}\}$ is a set of generating cofibrations.
  \end{theorem}
  \begin{proof}
    This is the main result of \cite{lafont2010folk}.
  \end{proof}
  \begin{paragr}\label{paragr:folkms}
    We refer to the model structure of the previous theorem as the \emph{Folk
      model structure} on $\omega\Cat$. Data of this model structure will often
    be referred to by using the adjective folk, \emph{e.g.\ folk cofibration}. From now on, unless otherwise explicitly specified, we will always consider that $\omega\Cat$ is equipped with this model structure. In particular, $\Ho(\omega\Cat)$ will always be the localization of $\omega\Cat$ with respect to the class of equivalences of $\omega$\nbd{}categories.
  \end{paragr}
  \begin{proposition}
    An $\omega$\nbd{}category is folk cofibrant if and only if it is free.
  \end{proposition}
  \begin{proof}
    The fact that every free $\omega$\nbd{}category is cofibrant follows immediately from the fact that the $i_n : \sS_{n-1} \to \sD_n$ are cofibrations and that every $\omega$\nbd{}category $C$ is the colimit of the canonical diagram
    \[
    \emptyset = C_{\leq -1} \to C_{\leq 0} \to \cdots \to C_{\leq n} \to C_{\leq n+1} \cdots
    \]
    For the converse, see \cite{metayer2008cofibrant}.
  \end{proof}
  \section{Polygraphic homology}
  \begin{paragr}
    Let $C$ be an $\omega$\nbd{}category. We define a chain complex in non-negative degree $\lambda(C)$ in the following way:
        \begin{itemize}
    \item[-] for $n \in \mathbb{N}$, $\lambda(C)_n$ is the abelian group obtained by quotienting the free abelian group $\mathbb{Z}C_n$ by the congruence generated by the relations
      \[
      x\ast_ky \sim x+y
      \]
      for all $x,y \in C_n$ that are $k$\nbd{}composable,
    \item[-] the differential $\partial : \lambda(C)_n \to \lambda(C)_{n-1}$ is induced by the map
      \[
      \begin{aligned}
        \mathbb{Z}C_n &\to \mathbb{Z}C_{n-1} \\
        x \in C_n &\mapsto t(x)-s(x).
      \end{aligned}
      \]
      \end{itemize}
      The axioms of $\omega$\nbd{}categories imply that $\partial \circ \partial = 0$. 
    Now let $f : C \to D$ be an $\omega$\nbd{}functor. The map
    \[
    \begin{aligned}
      \mathbb{Z}C_n &\to \mathbb{Z}D_{n}\\
      x \in C_n &\mapsto f(x)
      \end{aligned}
    \]
    induces a map
    \[
    \lambda(f)_n : \lambda(C)_n \to \lambda(D)_n.
    \]
    Since $f$ commutes with source and target, we obtain a morphism of chain complexes $(\lambda(f)_n)_{n\in \mathbb{N}}$. This defines a functor
      \[
      \lambda : \omega\Cat \to \Ch,
      \]
      where $\Ch$ is the category of chain complexes in non-negative degree, which we call the \emph{abelianization functor}. 
  \end{paragr}
  \begin{lemma}\label{lemma:abelpol}
    Let $C$ be a \emph{free $\omega$\nbd{}category} and let $\Sigma=(\Sigma_n)_{n\mathbb{N}}$ be its basis. Then for every $n\in\mathbb{N}$, $\lambda(C)_n$ is isomorphic to the free abelian group $\mathbb{Z}\Sigma_n$.
  \end{lemma}
  \begin{proof}
    Let $G$ be an abelian group. For every $n \in \mathbb{N}$, we define an $n$\nbd{}category $B^nG$ with:
    \begin{itemize}
    \item[-] $(B^nG)_{k}$ is a singleton set for every $k < n$,
    \item[-] $(B^nG)_n = G$
    \item[-] for all $x$ and $y$ in $G$ and $i<n$,
      \[x \ast_i y := x +y.\]
    \end{itemize}
    It is straightforward to check that this defines an $n$\nbd{}category. Note that when $n=1$, the previous definition would still make sense without the hypothesis that $G$ be abelian, but for $n\geq 2$ this hypothesis is necessary because of the Eckmann--Hilton argument. For $n=0$, we only needed that $G$ was a set.

    This defines a functor
    \[
    \begin{aligned}
      B^n : \Ab &\to n\Cat\\
      G &\mapsto B^nG,
    \end{aligned}
    \]
    which is easily seen to be right adjoint to the functor
    \[
    \begin{aligned}
      n\Cat &\to \Ab\\
      C &\mapsto \lambda(C)_n.
    \end{aligned}
    \]
    Now, if $C$ is an $\omega$\nbd{}category then $\lambda(C_{\leq n })_n=\lambda(C)_n$ and if $C$ is free with basis $\Sigma=(\Sigma_n)_{n\in \mathbb{N}}$, then for every abelian group $G$ there is a natural isomorphism
    \[
    \Hom_{n\Cat}(C_{\leq n},B^nG) \simeq \Hom_{\Set}(\Sigma_n,\vert G \vert ),
    \]
    where $|G|$ is the underlying set of $G$. Altogether, we have
    \[
    \begin{aligned}
      \Hom_{\Ab}(\lambda(C)_n,G) &\simeq \Hom_{\Ab}(\lambda(C_{\leq n})_n,G)\\
      &\simeq \Hom_{n\Cat}(C_{\leq n},B^nG)\\
      &\simeq \Hom_{\Set}(\Sigma_n,\vert G \vert)\\
      &\simeq \Hom_{\Ab}(\mathbb{Z}\Sigma_n,G).\qedhere
    \end{aligned}
    \]
    \end{proof}
  \begin{lemma}\label{lemma:adjlambda}
    The functor $\lambda$ is a left adjoint.
  \end{lemma}
  \begin{proof}
    The category of chain complexes is equivalent to the category $\omega\Cat(\Ab)$ of $\omega$\nbd{}categories internal to abelian groups (see \cite[Theorem 3.3]{bourn1990another}) and with this identification, the functor $\lambda : \omega\Cat \to \omega\Cat(\Ab)$ is nothing but the left adjoint of the canonical forgetful functor $\omega\Cat(\Ab) \to \omega\Cat$.
  \end{proof}
  \begin{paragr}
    Let $u, v : C \to D$ be two $\omega$\nbd{}functors. An \emph{oplax transformation} $\alpha$ from $u$ to $v$ consists of the following data:
    \begin{itemize}
    \item[-] for every $0$-cell $x$ of $C$, a $1$-cell of $D$
      \[
      \alpha_x : u(x) \to v(x),
      \]
      \item[-]for every $n$\nbd{}cell of $x$ of $C$ with $n>0$, an $(n+1)$\nbd{}cell of $D$
    \[
    \alpha_x : \alpha_{t_{n-1}(x)}\ast_{n-1}\cdots\ast_1\alpha_{t_0(x)}\ast_0u(x) \to v(x)\ast_0\alpha_{s_0(x)}\ast_1\cdots\ast_{n-1}\alpha_{s_{n-1}(x)}
    \]
    subject to the following axioms:
    \begin{enumerate}
    \item for every $n$\nbd{}cell $x$,
      \[\alpha_{1_x}=1_{\alpha_x},\]
    \item for all $0\leq k < n$, for all $n$\nbd{}cells $x$ and $y$ that are $k$\nbd{}composable,
      \[
      \begin{multlined}
      \alpha_{x \ast_k y}={\left(v(t_{k+1}(x))\ast_0\alpha_{s_0(x)}\ast_1\cdots\ast_{n-1}\alpha_{s_{n-1}(x)}\ast_k\alpha_y\right)}\\
            {\ast_{k+1}\left(\alpha_{t_{n-1}(x)}\ast_{n-1}\cdots\ast_1\alpha_{t_0(x)}\ast_0u(s_{k+1}(y))\right)}.
            \end{multlined}
      \]
      \end{enumerate}
    \end{itemize}
    We use the notation $\alpha : u \Rightarrow v$ to say that $\alpha$ is an oplax transformation from $u$ to $v$.
  \end{paragr}
  \begin{paragr}
    Let
    \[
    \begin{tikzcd}
      B \ar[r,"f"] & C \ar[r,shift left,"u"]  \ar[r,shift right,"v"']&D \ar[r,"g"] &E
    \end{tikzcd}
    \]
    be a diagram in $\omega\Cat$ and $\alpha : u \Rightarrow v$ be an oplax transformation. 
    The data of
    \[
    (g\star \alpha)_x := g(\alpha_x)
    \]
    for each cell $x$ of $C$ (resp. 
    \[
    (\alpha \star f)_x :=\alpha_{f(x)}
    \]
    for each cell $x$ of $B$) defines an oplax transformation from $g u$ to $g
    v$ (resp.\ $u f$ to $v f$) which we denote by $g\star \alpha$ (resp.\ $\alpha \star f$).
  \end{paragr}
   \begin{lemma}\label{lemma:abeloplax}
    Let $u, v : C \to D$ be two $\omega$\nbd{}functors. If there is an oplax transformation $\alpha : u \Rightarrow v$, then there is a homotopy of chain complexes from $\lambda(u)$ to $\lambda(v)$.
   \end{lemma}
   \begin{proof}
     For every $n$\nbd{}cell $x$ of $C$ (resp.\ $D$), let us use the notation $[x]$ for the image of $x$ in $\lambda(C)_n$ (resp.\ $\lambda(D)_n$).

  Let $h_n$ be the map
     \[
     \begin{aligned}
       h_n : \lambda(C)_n &\to \lambda(D)_{n+1}\\
       [x] & \mapsto [\alpha_x].
       \end{aligned}
     \]
     The definition of oplax transformations implies that $h_n$ is linear and that for every $n$\nbd{}cell $x$ of $C$,
     \[
     \partial (h_n(x)) + h_{n-1}(\partial(x)) = [v(x)] - [u(x)].
     \]
     Details are left to the reader. 
   \end{proof}
   \begin{paragr}
     Recall that the category of chain complexes in non-negative degree $\Ch$
     has a cofibrantly generated model structure where:
     \begin{itemize}
     \item[-] the weak equivalences are the quasi-isomorphisms, i.e.\ the morphisms of chain complexes that induce an isomorphism on homology groups,
              \item[-] the cofibrations are the morphisms of chain complexes $f: X\to Y$ such that for every $n\geq 0$, $f_n : X_n \to Y_n$ is a monomorphism with projective cokernel,
     \item[-] the fibrations are the morphisms of chain complexes $f : X \to Y$ such that for every $n>0$, $f_n : X_n \to Y_n$ is an epimorphism.
     \end{itemize}
     (See for example \cite[Section 7]{dwyer1995homotopy}.)
     From now on, we will implicitly consider that the category $\Ch$ is equipped with this model structure. 
   \end{paragr}
   \begin{proposition}
     The functor $ \lambda : \omega\Cat \to \Ch$ is left Quillen.
   \end{proposition}
   \begin{proof}
     The fact that $\lambda$ is a left adjoint is Lemma \ref{lemma:adjlambda}.

     A simple computation using Lemma \ref{lemma:abelpol} shows that for every $n\in \mathbb{N}$,
     \[
     \lambda(i_n) : \lambda(\sS_{n-1}) \to \lambda(\sD_{n})
     \]
     is a monomorphism with projective cokernel. This shows that $\lambda$ preserves cofibrations.

     Then, we know from \cite[Sections 4.6 and 4.7]{lafont2010folk} and \cite[Remarque B.1.16]{ara2016joint} (see also \cite[Paragraph 3.11]{ara2019folk}) that there exists a set of generating trivial cofibrations $J$ of the Folk model structure on $\omega\Cat$ such that every $j : X \to Y$ in $J$ satisfies the following conditions:
     \begin{itemize}
     \item[-] there exists $r : Y \to X$ such that $r\circ j = 1_X$,
       \item[-] there exists an oplax transformation $\alpha : j\circ r \Rightarrow 1_Y$.
     \end{itemize}
     From Lemma \ref{lemma:abeloplax}, we conclude that $\lambda$ preserves trivial cofibrations. 
   \end{proof}
     The previous proposition leads the following definition:
   \begin{definition}\label{def:polhom}
     We define the \emph{polygraphic homology functor}
     \[
     \sH^{\mathrm{pol}} : \Ho(\omega\Cat) \to \Ho(\Ch)
     \]
     as the total left derived functor of $\lambda : \omega\Cat \to \Ch$.
     \end{definition}
   \section{Nerve of \texorpdfstring{$\omega$}{\textomega}-categories and the comparison map}
     \begin{paragr}\label{paragr:simpset}
    We denote by $\Delta$ the category whose objects are the finite non-empty totally ordered sets $[n]=\{0<\cdots<n\}$ and whose morphisms are the non-decreasing maps. For $n \in \mathbb{N}$ and $0\leq i\leq n$, we denote by
    \[
    \delta^i : [n-1] \to [n] 
    \]
    the only injective increasing map whose image does not contain $i$.

    The category $\Psh{\Delta}$ of simplicial sets is the category of presheaves on $\Delta$. For a simplicial set $X$, we use the notations
    \[
    \begin{aligned}
      X_n &:= X([n]) \\
      \partial_i &:= X(\delta^i): X_n \to X_{n\shortminus 1}.
    \end{aligned}
    \]
    Elements of $X_n$ are referred to as \emph{$n$\nbd{}simplices of $X$}. 
     \end{paragr}
       \begin{paragr}
From now on, we will consider that the category $\Psh{\Delta}$ is equipped with the model structure defined by Quillen in \cite{quillen1967homotopical}. A \emph{weak equivalence of simplicial sets} is a weak equivalence for this model structure.
  \end{paragr}
   \begin{paragr}
     We denote by $\Or : \Delta \to \omega\Cat $ the cosimplicial object introduced by Street in \cite{street1987algebra}. The $\omega$\nbd{}category $\Or_n$ is the \emph{$n$\nbd{}oriental}. For a definition and basic properties of this cosimplicial object we refer to \emph{op.~cit.}, \cite{steiner2004omega} and \cite[Chapitre 7]{ara2016joint}.

     For every $n\in \mathbb{N}$, the $\omega$\nbd{}category $\Or_n$ is free and the set of generating $k$\nbd{}cells is canonically isomorphic to the set of increasing sequences \[0 \leq i_0 < i_1 < \cdots < i_{k} \leq n.\] We will denote such a generating cell by $\langle i_0i_1\cdots i_{k} \rangle$. In particular, $\Or_n$ is an $n$\nbd{}category and it has a unique generating $n$\nbd{}cell, namely $\langle 012\cdots n\rangle$, which we call the \emph{principal cell} of $\Or_n$.
   \end{paragr}
   Here are some pictures in low dimension:
\[
  \Or_0 = \langle 0 \rangle,
  \]
  \[
  \Or_1=\begin{tikzcd}
    \langle 0 \rangle \ar[r,"\langle 01 \rangle"] &\langle 1 \rangle,
    \end{tikzcd}
  \]
  \[
  \Or_2=
  \begin{tikzcd}
    &\langle 1 \rangle \ar[rd,"\langle 12 \rangle"]& \\
    \langle 0 \rangle \ar[ru,"\langle 01 \rangle"]\ar[rr,"\langle 02 \rangle"',""{name=A,above}]&&\langle 2 \rangle.
    \ar[Rightarrow,from=A,to=1-2,"\langle 012 \rangle"]
    \end{tikzcd}
  \]
  \begin{paragr}\label{paragr:nerve}
    For every $\omega$\nbd{}category $X$, the \emph{nerve of $X$} is the simplicial set $N_{\omega}(X)$ defined as
    \[
    \begin{aligned}
     N_{\omega}(X) : \Delta^{\mathrm{op}} &\to \Set\\
      [n] &\mapsto \Hom_{\omega\Cat}(\Or_n,X).
      \end{aligned}
    \]
   By post-composition, this yields a functor
  \[
  \begin{aligned}
  N_{\omega} : \omega\Cat &\to \Psh{\Delta} \\
  X &\mapsto N_{\omega}(X).
  \end{aligned}
  \]
  Note that when $X$ is a $1$-category, $N_{\omega}(X)$ is canonically isomorphic to the usual nerve of $X$, that is, the simplicial set
      \[
    \begin{aligned}
     \Delta^{\mathrm{op}} &\to \Set\\
      [n] &\mapsto \Hom_{\Cat}([n],X),
      \end{aligned}
    \]
    where $[n]$ is seen as a $1$-category.
    
    By the usual Kan extension technique, $\Or : \Delta \to \omega\Cat$ can be extended to a functor \[c_{\omega} : \Psh{\Delta} \to \omega\Cat,\]
    which is left adjoint to the nerve functor $N_{\omega}$.
  \end{paragr}
  \begin{lemma}\label{lemma:nervehomotopical}
The nerve functor $N_{\omega} : \omega\Cat \to \Psh{\Delta}$ sends the equivalences of $\omega$\nbd{}categories to weak equivalences of simplicial sets.    
  \end{lemma}
  In particular, this means that the nerve functor induces a functor
  \[
  \overline{N}_{\omega}: \Ho(\omega\Cat) \to \Ho(\Psh{\Delta}).
  \]
  \begin{proof}
    Since every $\omega$\nbd{}category is fibrant for the Folk model structure
    \cite[Proposition 9]{lafont2010folk}, it follows from Ken Brown's Lemma
    \cite[Lemma 1.1.12]{hovey2007model} that it suffices to show that the nerve
    sends the folk trivial fibrations to weak equivalences of simplicial sets.
    In particular, it suffices to show the stronger condition that the nerve
    sends the folk trivial fibrations to trivial fibrations of simplicial sets.

    By adjunction, this is equivalent to showing that the functor $c_{\omega} :
    \Psh{\Delta} \to \omega\Cat$ sends the cofibrations of simplicial sets to folk cofibrations. Since $c_{\omega}$ is cocontinuous and the cofibrations of simplicial sets are generated by the canonical inclusions
    \[
    \partial \Delta_n \to \Delta_n
    \]
    for $n \in \mathbb{N}$, it suffices to show that $c_{\omega}$ sends these inclusions to folk cofibrations.

    Now, it follows from any reference on orientals previously cited that the image of the inclusion $\partial \Delta_n \to \Delta_n$ by $c_{\omega}$ can be identified with the canonical inclusion
    \[
    (\Or_n)_{\leq n-1} \to \Or_n.
    \]
    Since $\Or_n$ is free, this last morphism is by definition a pushout of a coproduct of folk cofibrations, hence a folk cofibration.
  \end{proof}
  \begin{paragr}
    Let $X$ be a simplicial set. We denote by $K_n(X)$ the abelian group of \emph{$n$\nbd{}chains of $X$}, that is the free abelian group on $X_n$, and by $\partial : K_n(X) \to K_{n \shortminus 1}(X)$ the linear map defined for $x \in X_n$ by
    \[
    \partial(x)=\sum_{0\leq i \leq n}(\shortminus 1)^{i}\partial_i(x).
    \]
    It follows from the simplicial identities \cite[Section I.1]{goerss2009simplicial} that $\partial \circ \partial = 0$ and thus, the previous data defines a chain complex $K(X)$. 
    This canonically defines a functor
    \[
    K : \Psh{\Delta} \to \Ch
  \]
  in the expected way.
  \end{paragr}
  \begin{paragr}
    Recall that an $n$\nbd{}simplex $x$ of a simplicial set $X$ is \emph{degenerate} if there exists a surjective non-decreasing map $\varphi : [n] \to [k]$ with $k<n$ and a $k$\nbd{}simplex $x'$ of $X$ such that $X(\varphi)(x')=x$. For every $n \in \mathbb{N}$, we define $D_n(X)$ as the subgroup of $K_n(X)$ spanned by the degenerate $n$\nbd{}simplices.

    We denote by $\kappa_n(X)$ the abelian group of \emph{normalized chain complex},
    \[
    \kappa_n(X)=K_n(X)/{D_n(X)}.
    \]
    It follows from the simplicial identities that $\partial(D_n(X))\subseteq D_{n\shortminus 1}(X)$ and we denote by
    \[\partial : \kappa_n(X) \to \kappa_{n\shortminus 1}(X)\]
    the map induced by the differential of $K(X)$. This data defines a chain complex $\kappa(X)$ that we call the \emph{normalized chain complex of $X$}. This yields a functor
    \[
    \kappa : \Psh{\Delta} \to \Ch.
    \]
  \end{paragr}
  \begin{lemma}\label{lemma:normalizedquillen}
    The functor $\kappa : \Psh{\Delta} \to \Ch$ is left Quillen and sends the weak equivalences of simplicial sets to quasi-isomorphisms.
  \end{lemma}
  \begin{proof}
    From the Dold--Kan equivalence, we know that $\Ch$ is equivalent to the category $\Ab(\Delta)$ of simplicial abelian groups. With this identification the functor $\kappa : \Psh{\Delta} \to \Ch$ is left adjoint of the canonical forgetful functor
    \[
    U : \Ch \simeq \Ab(\Delta) \to \Psh{\Delta}
    \]
    induced by the forgetful functor from abelian groups to sets. It follows
    from \cite[Lemma 2.9 and Corollary 2.10]{goerss2009simplicial} that $U$ is
    right Quillen, hence $\kappa$ is left Quillen. The fact that $\kappa$ preserves weak equivalences follows from Ken Brown's Lemma \cite[Lemma 1.1.12]{hovey2007model} and the fact that all simplicial sets are cofibrant. 
    \end{proof}
  \begin{lemma}\label{lemma:abelor}
    The triangle of functors
    \[
    \begin{tikzcd}
      \Psh{\Delta} \ar[r,"c_{\omega}"] \ar[dr,"\kappa"']& \omega\Cat\ar[d,"\lambda"]\\
      &\Ch
    \end{tikzcd}
    \]
    is commutative (up to a canonical isomorphism).
  \end{lemma}
  \begin{proof}
    Since all the functors involved are cocontinuous, it suffices to show that the triangle is commutative when we pre-compose it by the Yoneda embedding $\Delta \to \Psh{\Delta}$. This property follows straightforwardly from the description of the orientals in \cite{steiner2004omega}.
  \end{proof}
  \begin{paragr}
    From Lemma \ref{lemma:abelor}, the co-unit of the adjunction $c_{\omega} \dashv N_{\omega}$ induces a natural transformation
    \[
    \kappa N_{\omega} \simeq \lambda c_{\omega} N_{\omega} \Rightarrow \lambda.
    \]
    From Lemma \ref{lemma:nervehomotopical}, Lemma \ref{lemma:normalizedquillen}, Remark \ref{rem:homotopicalisleftder} and the universal property of left derivable functors, we obtain a natural transformation
    \[
    \overline{\kappa} \overline{N}_{\omega} \Rightarrow \sH^{\mathrm{pol}},
    \]
which we refer to as the \emph{canonical comparison map}.  Following Remark
\ref{remark:homotopicalfunctor}, for every $\omega$\nbd{}category $C$ this canonical
comparison map reads
\[
  \kappa N_{\omega}(C)\to\sH^{\mathrm{pol}}(C).
\]
  \end{paragr}
  \section{The case of contractible \texorpdfstring{$\omega$}{\textomega}-categories}
   \begin{paragr}
     For every $\omega$\nbd{}category $C$, we denote by
     \[
     p_C : C \to \sD_0
     \]
     the unique morphism from $C$ to $\sD_0$ (which is a terminal object of $\omega\Cat$).
     \end{paragr}
    \begin{definition}\label{def:contractiblecat}
      An $\omega$\nbd{}category $C$ is \emph{contractible} if there exists a $0$-cell $x$ of $C$ and an oplax transformation
      \[
      \begin{tikzcd}[column sep=huge,row sep=huge]
        C \ar[dr,"\mathrm{id}_C"',""{name=A,above}] \ar[r,"p_C"] & \sD_0 \ar[d,"\langle x \rangle"] \\
          &C.
          \ar[from=A,to=1-2,"\alpha",Rightarrow,shorten >= 1ex]
        \end{tikzcd}
      \]
    \end{definition}
    For later reference, we put here the following lemma.
    \begin{lemma}\label{lemma:terminalobject}
      Let $C$ be a $1$\nbd{}category. If $C$ has a terminal object, then it is contractible.
    \end{lemma}
    \begin{proof}
      Let $x$ be the terminal object of $C$. For every object $y$ of $C$, the
      canonical arrow $y \to x$ of $C$ defines a natural transformation
      \[
        \begin{tikzcd}
          C \ar[r,"p_C"]  \ar[dr,"\mathrm{id}_C"',""{name=A,above}] & \sD_0 \ar[d,"\langle x \rangle"]\\
          & C.
          \ar[from=A,to=1-2,Rightarrow]
          \end{tikzcd}
        \]
        The result follows then from the fact that any natural transformation between $1$\nbd{}functors
        obviously defines an oplax transformation.
      \end{proof}
     
      For the next lemma, notice that an immediate computation using Lemma \ref{lemma:abelpol} shows that $\lambda(\sD_0)$ is canonically isomorphic to $\mathbb{Z}$ considered as a chain complex concentrated in degree $0$. 
    
    \begin{lemma}\label{lemma:abelcontractible}
      Let $C$ be a contractible $\omega$\nbd{}category. The morphism of chain complexes
      \[
      \lambda(p_C) : \lambda(C) \to \lambda(\sD_0)\simeq \mathbb{Z}
      \]
      is a quasi-isomorphism.
    \end{lemma}
    \begin{proof}
      This follows immediately from Lemma \ref{lemma:abeloplax}.
    \end{proof}
    
    \begin{lemma}\label{lemma:liftingoplax}
      Let
      \[
      \begin{tikzcd}
        C' \ar[r,"f_{\epsilon}'"] \ar[d,"u"] & D' \ar[d,"v"]\\
        C \ar[r,"f_{\epsilon}"] & D
      \end{tikzcd}
      \]
      be commutative squares in $\omega\Cat$ for $\epsilon\in\{0,1\}$.

      If $C'$ is a free $\omega$\nbd{}category and $v$ is a folk trivial fibration, then for every oplax transformation \[\alpha : f_0 \Rightarrow f_1,\] there is an oplax transformation \[\alpha' : f_0' \Rightarrow f_1'\] such that
      \[
      v \star \alpha' = \alpha \star u.
      \]
    \end{lemma}
    \begin{proof}
      We denote by $\otimes$ the \emph{Gray monoidal product} (see for example \cite[Appendice A]{ara2016joint}) on the category $\omega\Cat$. Recall that the unit of this monoidal product is the $\omega$\nbd{}category $\sD_0$.
      
      From \cite[Appendice B]{ara2016joint}, we know that given two
      $\omega$\nbd{}functors $u, v : C \to D$, the set of oplax transformations from $u$ to $v$ is in bijection with the set of functors $\alpha : \sD_1 \otimes C \to D$ such that the diagram
      \[
      \begin{tikzcd}
            (\sD_0\amalg \sD_0) \otimes C \simeq  C \amalg C \ar[d,"i_1 \otimes C"'] \ar[dr,"{\langle u, v \rangle}"] &\\
        \sD_1 \otimes C \ar[r,"\alpha"'] & D,
      \end{tikzcd}
      \]
      where $i_1 : \sD_0 \amalg \sD_0 \simeq \sS_0 \to \sD_1$ is the morphism introduced in \ref{paragr:defglobe}, is commutative. We use the same letter to denote an oplax transformation and the functor $\sD_1 \otimes C \to D$ associated to it.

      Moreover, for an $\omega$\nbd{}functor $f: B \to C$, the oplax transformation $\alpha \star f$ is represented by the functor \[\begin{tikzcd} \sD_1 \otimes B \ar[r,"\sD_1 \otimes f"] & \sD_1 \otimes C \ar[r,"\alpha"]& D\end{tikzcd}\] and for an $\omega$\nbd{}functor $g : D \to E$, the oplax transformation $g \star \alpha$ is represented by the functor
      \[\begin{tikzcd}\sD_1 \otimes C \ar[r,"\alpha"]& D \ar[r,"g"] & E.\end{tikzcd}\]
      
      Using this way of representing oplax transformations, the hypotheses of the present lemma yield the following commutative square
      \[
      \begin{tikzcd}
        (\sD_0 \amalg \sD_0)\otimes C' \ar[d,"{i_1\otimes C'}"'] \ar[rr,"{\langle f'_0, f_1' \rangle}"] && D' \ar[d,"v"] \\
        \sD_1\otimes C'\ar[r,"\sD_1 \otimes u"'] & \sD_1\otimes C \ar[r,"\alpha"] & D.
        \end{tikzcd}
      \]
      Since $i_1$ is a folk cofibration and $C'$ is cofibrant, it follows that
      the left vertical morphism of the previous square is a folk cofibration
      (see \cite[Proposition 5.1.2.7]{lucas2017cubical} or \cite{ara2019folk}).
      By hypothesis, $v$ is a folk trivial fibration and thus the above square admits a lift
      \[
      \alpha' : \sD_1\otimes C' \to D'.
      \]
      The commutativity of the two induced triangle shows what we needed to prove.
    \end{proof}
    \begin{proposition}\label{prop:comparisoncontractible}
      Let $C$ be an $\omega$\nbd{}category. If $C$ is contractible, then the canonical comparison map
      \[
      \kappa N_{\omega}(C) \to \sH^{\mathrm{pol}}(C)
      \]
      is an isomorphism of $\Ho(\Ch)$. More precisely, this morphism can be identified with the identity morphism $\mathrm{id}_{\mathbb{Z}} : \mathbb{Z} \to \mathbb{Z}$ (where $\mathbb{Z}$ is seen as an object of $\Ho(\Ch)$ concentrated in degree $0$). 
    \end{proposition}
    \begin{proof}
      Consider first the case when $C$ is cofibrant for the Folk model structure. It follows respectively from Lemma \ref{lemma:abelcontractible} and \cite[Corollaire A.13]{ara2020theoreme} that the morphisms
      \[
      \sH^{\mathrm{pol}}(C) \to \sH^{\mathrm{pol}}(\sD_0)
      \]
      and 
      \[
      \kappa N_{\omega}(C) \to \kappa N_{\omega}(\sD_0),
      \]
      induced by the canonical morphism $p_C : C \to \sD_0$, are isomorphisms of $\Ho(\Ch)$. Moreover, it is straightforward to check that the canonical comparison map
      \[
      \kappa N_{\omega}(\sD_0) \to \sH^{\mathrm{pol}}(\sD_0)
      \]
      is an isomorphism $\Ho(\Ch)$ (and can be identified with the identity morphism $\mathrm{id}_{\mathbb{Z}} : \mathbb{Z} \to \mathbb{Z}$, where $\mathbb{Z}$ is seen as a chain complex concentrated in degree $0$). From the naturality square
     \[
      \begin{tikzcd}
        \kappa N_{\omega}(C)\ar[d] \ar[r]  &\sH^{\mathrm{pol}}(C) \ar[d]\\
        \kappa N_{\omega}(\sD_0) \ar[r] & \sH^{\mathrm{pol}}(\sD_0)
      \end{tikzcd}
      \]
      we deduce that the top arrow is an isomorphism of $\Ho(\Ch)$.

      In the general case, we know by definition of contractible
      $\omega$\nbd{}categories that there exist a $0$\nbd{}cell $x$ of $C$ and an
      oplax transformation $\alpha$ of the form
        \[
      \begin{tikzcd}
        C \ar[dr,"\mathrm{id}_C"',""{name=A,above}] \ar[r,"p_C"] & \sD_0 \ar[d,"\langle x \rangle"] \\
          &C.
          \ar[from=A,to=1-2,"\alpha",Rightarrow,shorten >= 1ex]
        \end{tikzcd}
      \]
      Now let us choose any folk trivial
      fibration $ u : P \to C$ with $P$ cofibrant. By definition, folk trivial
      fibrations have the right lifting property to $i_0 : \sS_{-1} \to \sD_0$, which
      means exactly that they are surjective on $0$\nbd{}cells. In particular,
      there exists a $0$\nbd{}cell $x'$ of $P$ such that $u(x')=x$. It follows
      then from Lemma \ref{lemma:liftingoplax} that there exists an oplax
      transformation
           \[
      \begin{tikzcd}
        P \ar[dr,"\mathrm{id}_P"',""{name=A,above}] \ar[r,"p_P"] & \sD_0 \ar[d,"\langle x' \rangle"] \\
          &P,
          \ar[from=A,to=1-2,"\alpha'",Rightarrow,shorten >= 1ex]
        \end{tikzcd}
      \]
      which proves that $P$ is also contractible. 

      Finally, consider the naturality square
      \[
      \begin{tikzcd}
        \kappa N_{\omega}(P)\ar[d] \ar[r]  &\sH^{\mathrm{pol}}(P) \ar[d]\\
        \kappa N_{\omega}(C) \ar[r] & \sH^{\mathrm{pol}}(C)
      \end{tikzcd}
      \]
      induced by $u$. Since $P$ is contractible and folk cofibrant, we have already proved that
      the top arrow is an isomorphism. Since $u$ is a weak equivalence for the
      Folk model structure, the right vertical morphism is an isomorphism and it
      follows from Lemma \ref{lemma:nervehomotopical} and Lemma
      \ref{lemma:normalizedquillen} that the left vertical arrow is also an
      isomorphism. Hence, the bottom arrow is an isomorphism.
    \end{proof}
\section{The folk homotopy colimit theorem}
    \begin{paragr}
      For every object $a$ of a $1$\nbd{}category $A$, we denote by $A/a$ the slice category over $a$. That is, an object of $A/a$ is a pair $(b,p : b \to a)$ where $b$ is an object of $A$ and $p$ is an arrow of $A$, and an arrow $(b,p) \to (b',p')$ of $A/a$ is an arrow $q : b \to b'$ of $A$ such that $p'\circ q = p$.
      
      We denote by
      \[
      \begin{aligned}
        \pi_a : A/a &\to A\\
        (b,p) &\mapsto b
      \end{aligned}
      \] the canonical forgetful functor.
    \end{paragr}
    \begin{definition}
      Let $A$ be a $1$\nbd{}category, $a$ be an object of $A$, $X$ be an
      $\omega$\nobreakdash-category and $f : X \to A$ be an
      $\omega$\nbd{}functor. We define the $\omega$\nbd{}category $X/a$ and the $\omega$\nbd{}functor \[f/a : X/a \to A/a\] as the following pullback
      \[
      \begin{tikzcd}
        X/a \ar[r]\ar[dr, phantom, "\lrcorner", very near start] \ar[d,"f/a"'] & X \ar[d,"f"] \\
        A/a \ar[r,"\pi_a"] &A.
      \end{tikzcd}
      \]
    \end{definition}
    \begin{paragr}
      More explicitly, an $n$\nbd{}arrow of $X/a$ is a pair $(x,p)$ where $x$ is
      an $n$\nbd{}arrow of $X$ and $p$ is an arrow of $A$ of the form
      \[ p : f(x) \to a \text{ if } n=0\]
      and
      \[p : f(t_0(x)) \to a \text{ if } n>0.\]
      \emph{From now on, we adopt the convention that $t_0(x)=x$ for every
        $0$\nbd{}cell $x$ of $X$. This unifies the cases $n=0$ and $n>0$ above.}
      
      For $n>0$, the source and target of an $n$\nbd{}arrow $(x,p)$ of $X/a$ are given by
      \[
      s((x,p))=(s(x),p) \text{ and } t((x,p))=(t(x),p).
      \]
      For $(x,p)$ an $n$\nbd{}arrow of $X/a$, we have
      \[
      (f/a)(x,p) = (f(x),p),
      \]
      and the canonical arrow $X/a \to X$ is simply expressed as
      \[
      (x,p) \mapsto x.
      \]
    \end{paragr}
    \begin{paragr}\label{paragr:unfoldfunctor}
     Let $f : X \to A$ be an $\omega$\nbd{}functor with $A$ a $1$-category. Every arrow $\beta : a \to a'$ of $A$ induces an $\omega$\nbd{}functor
      \[
      \begin{aligned}
        X/\beta : X/a &\to X/a'\\
        (x,p) & \mapsto (x, \beta \circ p).
      \end{aligned}
      \]
        This defines a functor
        \[
        \begin{aligned}
          X/{\shortminus} : A &\to \omega\Cat\\
          a &\mapsto X/a.
          \end{aligned}
        \]
           Moreover, for every arrow $\beta : a \to a'$ of $A$, the triangle
        \[
        \begin{tikzcd}[column sep=tiny]
          X/a \ar[dr]  \ar[rr,"X/\beta"] && X/a' \ar[dl] \\
          &X&
          \end{tikzcd}
        \]
        is commutative. By universal property, this induces a canonical arrow
        \[
        \colim_{a \in A}X/a \to X.
        \]
        Now, let
        \[
        \begin{tikzcd}[column sep=tiny]
          X \ar[rr,"g"] \ar[dr,"f"'] && X' \ar[dl,"f'"] \\
          &A&
          \end{tikzcd}
        \]
        be a commutative triangle in $\omega\Cat$ with $A$ a 1-category. For every $a~\in~\Ob(A)$, we define the $\omega$\nbd{}functor $g/a$ as 
        \[
        \begin{aligned}
         g/a : X/a &\to X'/a\\
          (x,p) &\mapsto (g(x),p).
        \end{aligned}
      \]
      This construction is natural in $a$, which means that for every arrow $\beta : a
      \to a'$ of $A$, the following square is commutative
      \[
        \begin{tikzcd}
          X/a \ar[r,"g/a"] \ar[d,"X/{\beta}"'] & X'/a \ar[d,"X'/{\beta}"] & \\
          X/{a'} \ar[r,"g/{a'}"'] & X'/{a'}.
        \end{tikzcd}
      \]
        In particular, $g$ induces an $\omega$\nbd{}functor
        \[
        \colim_{a \in A}X/a \to \colim_{a \in A}X'/a
        \]
        and the square
        \[
        \begin{tikzcd}
          \displaystyle\colim_{a \in A}X/a \ar[d] \ar[r,] & X \ar[d,"g"] \\
          \displaystyle\colim_{a \in A}X'/a \ar[r] & X'
        \end{tikzcd}
        \]
        is commutative.
    \end{paragr}
    \begin{lemma}\label{lemma:colimslice}
      Let $f : X \to A$ be an $\omega$\nbd{}functor with $A$ is a $1$\nbd{}category. The canonical arrow
      \[
      \colim_{a \in A}X/a \to X
       \]
       is an isomorphism.
    \end{lemma}
    \begin{proof}
      We have to show that the cocone
      \[
      (X/a \to X)_{a \in \Ob(A)}
      \]
      is colimiting.
      Let
      \[
      (g_a : X/a \to C)_{a \in \Ob(A)}
      \]
      be another cocone and let $x$ be an $n$\nbd{}arrow of $X$. Notice that the pair
      \[
      (x,1_{f(t_0(x))})
      \]
      is an $n$\nbd{}arrow of $X/f(t_0(x))$. We leave it to the reader to prove that the formula
      \[
      \begin{aligned}
        \phi : X &\to C \\
        x &\mapsto g_{f(t_0(x))}(x,1_{f(t_0(x))})
      \end{aligned}
      \]
      defines an $\omega$\nbd{}functor. This proves the existence part of the universal property.
      
      It is straightforward to check that for every $a \in \Ob(A)$ the triangle
      \[
      \begin{tikzcd}
        X/a\ar[dr,"g_a"']\ar[r] & X \ar[d,"\phi"] \\
        & C
      \end{tikzcd}
      \]
      is commutative. Now let $\phi' : X \to C$ be another $\omega$\nbd{}functor that makes the previous triangles commute and let $x$ be an $n$\nbd{}arrow of $X$. Since the triangle
      \[
      \begin{tikzcd}
        X/f(t_0(x)) \ar[dr,"g_{f(t_0(x))}"']\ar[r] & X \ar[d,"\phi'"] \\
        & C
      \end{tikzcd}
      \]
      is commutative, we necessarily have
      \[
      \phi'(x)=g_{f(t_0(x))}(x,1_{f(t_0(x))})
      \]
      which proves that $\phi'=\phi$.
    \end{proof} 
    \begin{lemma}\label{lemma:sliceisfree}
      If $X$ is free, then for every $a \in \Ob(A)$ the $\omega$\nbd{}category
      $X/a$ is free. More precisely, if we denote by $\Sigma^X_n$ the set of generating $n$\nbd{}cells of $X$, then the set of generating $n$\nbd{}cells of $X/a$ is 
      \[
      \Sigma^{X/a}_n = \{(x,p) \in (X/a)_n | x \in \Sigma^X_n\}.
      \]
    \end{lemma}
    \begin{proof}
      Remark first that for every $a \in \Ob(A)$, the map
      \[
      \pi_a : A/a \to A
      \]
      is a discrete Conduché $\omega$\nbd{}functor (Definition \ref{def:condfun}). Hence, from Lemma \ref{lemma:condfunpullback}, we deduce that the canonical map
      \[
      X/a \to X
      \]
      is also a discrete Conduché $\omega$\nbd{}functor. We conclude with Proposition \ref{prop:condfunfree}.
    \end{proof}
    \begin{paragr}
      Let $f : X \to A$ be as before and suppose that $X$ is free. Every arrow $\beta : a \to a'$ of $A$ induces a map:
      \[
      \begin{aligned}
        \Sigma^{X/a'}_n &\to \Sigma^{X/a}_n\\
        (x,p) &\mapsto (x,\beta\circ p).
      \end{aligned}
      \]
      This defines a functor
      \[
      \begin{aligned}
        \Sigma^{X/\shortminus}_n : A &\to \Set \\
        a &\mapsto \Sigma^{X/a}_n.
        \end{aligned}
      \]
    \end{paragr}
    \begin{lemma}\label{lemma:basisofslice}
      If $X$ is free, then there is a natural isomorphism
      \[
      \Sigma^{X/\shortminus}_n  \simeq \coprod_{x \in \Sigma^X_n}\Hom_{A}\left(f(t_0(x)),\shortminus\right).
      \]
    \end{lemma}
    \begin{proof}
      Let $a$ be an object of $A$. For each $x \in \Sigma_n^X$, there is a canonical map
      \[
      \begin{aligned}
          \Hom_{A}\left(f(t_0(x)),a\right) &\to \Sigma^{X/a}_n\\
        p &\mapsto (x,p)
      \end{aligned}
      \]
      that induces by universal property a map
      \[
      \coprod_{x \in \Sigma^X_n}\Hom_{A}\left(f(t_0(x)),a\right) \to \Sigma_n^{X/a}.
      \]
      The naturality in $a$ and the fact that this map is an isomorphism are obvious.
    \end{proof}
    \begin{proposition}\label{prop:sliceiscofibrant}
      Let $f : X \to A$ be an $\omega$\nbd{}functor with $A$ a 1-category and $X$ a free $\omega$\nbd{}category. The functor
      \[
      \begin{aligned}
        X/\shortminus : A &\to \omega\Cat\\
        a &\mapsto X/a
      \end{aligned}
      \]
      is a cofibrant object for the projective model structure on $\omega\Cat^{A}$ induced by the Folk model structure on $\omega\Cat$. 
    \end{proposition}
    \begin{proof}
      Recall that the set
      \[
       \{i_n: \sS_{n-1} \to \sD_n \vert n \in \mathbb{N} \}
      \]
      is a set a generating cofibrations for the Folk model structure on $\omega\Cat$.
      From Lemmas \ref{lemma:sliceisfree} and \ref{lemma:basisofslice} we deduce
      that for every $a \in \Ob(A)$ and every $n \in \mathbb{N}$, the canonical square
      \[
      \begin{tikzcd}
        \displaystyle\coprod_{x \in \Sigma^X_n}\coprod_{\Hom_A(f(t_0(x)),a)}\sS_{n-1} \ar[r] \ar[d] & (X/a)_{\leq n-1} \ar[d]\\
        \displaystyle\coprod_{x \in \Sigma^X_n}\coprod_{\Hom_A(f(t_0(x)),a)}\sD_n \ar[r]& (X/a)_{\leq n}
      \end{tikzcd}
      \]
      is a pushout square. It is straightforward to check that this square is
      natural in $a$, which means that for every $n \in
      \mathbb{N}$ we have a pushout
      square of $\omega\Cat^A$
      \[
        \begin{tikzcd}
          \displaystyle\coprod_{x \in \Sigma^X_n}\sS_{n-1}\otimes f(t_0(x)) \ar[d] \ar[r] &
          (X/\shortminus)_{\leq n -1} \ar[d] \\
          \displaystyle\coprod_{x \in \Sigma^X_n}\sD_{n}\otimes f(t_0(x)) \ar[r] &
          (X/\shortminus)_{\leq n} 
        \end{tikzcd}
      \]
      (see \ref{paragr:cofprojms} for notations). Hence, from Lemma
      \ref{lemma:cofprojms}, for every $n \geq0$ the morphism
      \[(X/-)_{\leq n-1} \to (X/-)_{\leq n}\]
      is a cofibration for the projective model structure on $\omega\Cat^A$. The
      result follows then from the fact that the colimit of
      \[
        \emptyset = (X/\shortminus)_{\leq -1} \to (X/\shortminus)_{\leq 0} \to
        \cdots \to (X/\shortminus)_{\leq n} \to \cdots
      \]
      is $X/\shortminus$.
    \end{proof}
    \begin{theorem}\label{thm:homcolfolk}
      Let $X$ be an $\omega$\nbd{}category, $A$ be a $1$-category and $f : X \to
      A$ be an $\omega$\nbd{}functor. The canonical arrow of $\Ho(\omega\Cat)$ 
      \[
      \hocolim_{a \in A}X/a \to \colim_{a \in A}X/a
    \]
    (see \ref{paragr:hocolim}) is an isomorphism. 
    \end{theorem}
    \begin{nb*}
      Note that in the previous theorem, we did \emph{not} suppose that $X$ was folk cofibrant.
      \end{nb*}
    \begin{proof}[Proof of Theorem \ref{thm:homcolfolk}]
      Let $P$ be a free $\omega$\nbd{}category and $g : P \to X$ be a trivial fibration for the Folk model structure. We have a commutative diagram in $\Ho(\omega\Cat)$ 
      \begin{equation}\label{comsquare}
      \begin{tikzcd}
        \displaystyle\hocolim_{a \in A}P/a \ar[d] \ar[r] & \displaystyle\colim_{a \in A}P/a \ar[d] \\
                \displaystyle\hocolim_{a \in A}X/a  \ar[r] & \displaystyle\colim_{a \in A}X/a 
      \end{tikzcd}
      \end{equation}
      where the vertical arrows are induced by the arrows
      \[
      g/a : P/a \to X/a.
      \]
      Since trivial fibrations are stable by pullback, $g/a$ is a trivial fibration. This proves that the left vertical arrow of square (\ref{comsquare}) is an isomorphism.
      
      Moreover, from paragraph \ref{paragr:unfoldfunctor} and Lemma \ref{lemma:colimslice}, we deduce that the right vertical arrow of (\ref{comsquare}) can be identified with the image of $g : P \to X$ in $\Ho(\omega\Cat)$ and hence is an isomorphism.

      Finally, from Proposition \ref{prop:sliceiscofibrant}, we deduce that the
      top horizontal arrow of (\ref{comsquare}) is an isomorphism. This proves the theorem. 
    \end{proof}
    \section{Homology of 1-categories}
    \begin{paragr}
      For a functor $d : A \to \Cat$ with $A$ a small category, we denote by
      $\int_A d$ the \emph{Grothendieck construction for $d$}. We refer to
      \cite[Section 3.1]{maltsiniotis2005theorie} for a definition and basic
      properties of this construction. Recall that there is a canonical morphism
      $\int_A d \to A$ as well a canonical morphism $\int_A d~\to~\colim_A (d)$,
      and that if $d$ is the functor $a \mapsto A/a$ then these two morphisms
      can be identified via the isomorphism $\displaystyle\colim_{a \in A} A/a \simeq A$.
    \end{paragr}
    \begin{lemma}\label{lemma:homcoltho}
      Let $A$ be a small category and consider the functor
      \[
      \begin{aligned}
        A &\to \omega\Cat\\
        a &\mapsto A/ {a}.
      \end{aligned}
      \]
      The canonical map
      \[
      \hocolim_{a \in A}N_{\omega}(A/ {a}) \to N_{\omega}(\hocolim_{a \in A}A/a)
      \]
      is an isomorphism of $\Ho(\Psh{\Delta})$.
    \end{lemma}
    \begin{proof}
      The proof is long and we divide it in several parts. Recall that we
      consider $\Cat$ as a full subcategory of $\omega\Cat$ (see Remark
      \ref{rem:1cat}). Given a functor $d : A \to \Cat$, we still denote by $d$ the functor obtained by post-composition
      \[
      \begin{tikzcd}
        A \ar[r,"d"] &\Cat \ar[hook,r] &\omega\Cat.
      \end{tikzcd}
      \]
      \begin{description}
      \item[Preliminaries]
        We say that a morphism $f : X \to Y$ of $\omega\Cat$ is a
        \emph{Thomason equivalence} when $N_{\omega}(f)$ is a weak equivalence
        of simplicial sets. We denote by $\W_{\omega}^{\mathrm{Th}}$ the class
        of Thomason equivalences. In order to avoid any confusion, we use the
        notation $\Ho^{\mathrm{Th}}(\omega\Cat)$ for the localization of
        $\omega\Cat$ with respect to the Thomason equivalences and the notation $\Ho^{\mathrm{folk}}(\omega\Cat)$ the localization of $\omega\Cat$ with respect to $\W_{\omega}^{\mathrm{folk}}$, the weak equivalences of the folk model structure (\ref{paragr:folkms}).
        
        Similarly, we denote by $\W_1^{\mathrm{Th}}$ the class of arrows of
        $\Cat$ whose elements are the Thomason equivalences (seen as arrows of $\omega\Cat$) and by $\Ho^{\mathrm{Th}}(\Cat)$ the localization of $\Cat$ with respect to $\W_1^{\mathrm{Th}}$. Note that $\W_1^{\mathrm{Th}}$ is indeed the class of weak equivalences of the model structure on $\Cat$ considered by Thomason in \cite{thomason1980cat}.
        
          Thus, we have defined three localizers: $(\Cat,\W_1^{\mathrm{Th}})$, $(\omega\Cat,\W_{\omega}^{\mathrm{Th}})$ and $(\omega\Cat,\W_{\omega}^{\mathrm{folk}})$. We have already seen that $(\omega\Cat,\W_{\omega}^{\mathrm{folk}})$ has homotopy colimits and from the existence of the Thomason model structure on $\Cat$ \cite{thomason1980cat}, we deduce that $(\Cat,\W_1^{\mathrm{Th}})$ has homotopy colimits. Although the existence of a model structure on $\omega\Cat$ with the Thomason equivalences as weak equivalences is not established (see \cite{ara2014vers} though), we will see later in this proof that $(\omega\Cat,\W_{\omega}^{\mathrm{Th}})$ has homotopy colimits. In order to explicitly distinguish the homotopy colimits in these three localizers, we use the self-explanatory notations:
          \[
          \hocolim_A^{\Cat,\mathrm{Th}},\quad \hocolim_A^{\omega\Cat,\mathrm{Th}} \text{ and } \hocolim_A^{\omega\Cat,\mathrm{folk}}.
          \]
          Similarly, we use the notations
          \[
          \colim_A^{\Cat} \text{ and } \colim_A^{\omega\Cat}
          \]
          to distinguish colimits in $\Cat$ and $\omega\Cat$.
          
    \item[Thomason's homotopy colimit theorem]
      Let us recall an important result due to Thomason
      \cite{thomason1979homotopy} (see also \cite[section
      3.1]{maltsiniotis2005theorie}): the functor \[\int_A : \Cat^A \to \Cat\]
      preserves the Thomason equivalences and the induced functor \[\overline{\int_A} : \Ho^{\mathrm{Th}}(\Cat^A) \to \Ho^{\mathrm{Th}}(\Cat)\] is a left adjoint to the diagonal functor (\ref{paragr:diagfun})
      \[\overline{k} : \Ho^{\mathrm{Th}}(\Cat) \to \Ho^{\mathrm{Th}}(\Cat^A).\] Hence, there is a canonical isomorphism of functors \[\overline{\int_A}\simeq \hocolim^{\Cat,\mathrm{Th}}_A.\]

      Now, since for every object $a$ of $A$ the category $A/a$ has a terminal object, it follows from \cite[Section 1,Corollary 2]{quillen1973higher} that the canonical morphism
      \[
      p_{A/a} : A/a \to \sD_0
      \]
      to the terminal category is a Thomason equivalence. In particular, if $d : A \to \Cat$ is the functor $a \mapsto A/a$, the induced map
      \begin{equation}\label{canmapintegral}
      \int_{A}d \to \int_A k(\sD_0), 
      \end{equation}
      where $k(\sD_0)$ is the constant diagram with value $\sD_0$, is an isomorphism of $\Ho^{\mathrm{Th}}(\Cat)$. A quick computation left to the reader shows that \[\int_A k(\sD_0) \simeq A\] and that the map \eqref{canmapintegral} can be identified with the canonical map
      \[
      \hocolim_{a \in A}^{\Cat,\mathrm{Th}}A/a \to \colim_{a \in A}^{\Cat} A/a,
      \]
      which is thus an isomorphism of $\Ho^{\mathrm{Th}}(\Cat)$.
    \item[Preservation of Thomason homotopy colimits]
      From \cite[Theorem 2.4 and 6.9]{gagna2018strict}, it follows that
      $N_{\omega} : \omega\Cat \to \Psh{\Delta}$ induces an equivalence of
      prederivators (here $\omega\Cat$ is equipped with the Thomason equivalences). Concretely to us, this implies that $(\omega\Cat,\W_{\omega}^{\mathrm{Th}})$ has homotopy colimits and that for every functor $d : A \to \omega\Cat$, the canonical map
      \[
      \hocolim_A^{\Psh{\Delta}}(N_{\omega}(d)) \to N_{\omega}(\hocolim_A^{\omega\Cat,\mathrm{Th}}(d))
      \]
      is an isomorphism of $\Ho(\Psh{\Delta})$. (For the reader unfamiliar with
      the theory of prederivators, the above result follows also from the fact that 
      Gagna's result together with \cite[Theorem
      1.8]{barwick2012characterization} imply that
      $N_{\omega}$ induces an equivalence between the associated
      $\infty$\nbd{}categories in the sense of Lurie \cite{lurie2009higher}.)
            
      Similarly, the usual nerve functor for $1$-categories $N_1 : \Cat \to \Psh{\Delta}$ induces an equivalence of prederivators and from the commutativity of the triangle
      \[
      \begin{tikzcd}
        \Cat \ar[rr,"i"]\ar[dr,"N_1"'] & & \omega\Cat\ar[dl,"N_{\omega}"]\\        
        &\Psh{\Delta}&,
      \end{tikzcd}
      \]
      and the fact that $i$ preserves the Thomason equivalences (by definition), we deduce that $i$ also induces an equivalence of prederivators. Hence, for every functor $d : A \to \Cat$, the canonical map
      \[
      \hocolim_A^{\omega\Cat,\mathrm{Th}}(d)\to\hocolim_A^{\Cat,\mathrm{Th}}(d)
      \]
      is an isomorphism of $\Ho^{\mathrm{Th}}(\omega\Cat)$. Consider now the commutative square in $\Ho^{\mathrm{Th}}(\omega\Cat)$:
        \[
        \begin{tikzcd}
          \displaystyle\hocolim_{a \in A}^{\Cat,\mathrm{Th}}A/a \ar[r] \ar[d] & \displaystyle \hocolim_{a \in A}^{\omega\Cat,\mathrm{Th}}A/a \ar[d]\\
          \displaystyle\colim_{a \in A}^{\Cat}A/a \ar[r]&\displaystyle\colim_{a \in A}^{\omega\Cat}A/a.
        \end{tikzcd}
        \]
        So far we have proved that the top horizontal arrow and the left vertical arrows are isomorphisms. Since the inclusion functor $\Cat \to \omega\Cat$ preserves colimits, the bottom horizontal arrow is also an isomorphism. This implies that the right vertical arrow is an isomorphism.
      \item[Comparing Folk and Thomason homotopy colimits]From Lemma \ref{lemma:nervehomotopical}, we have that $\W_{\omega}^{\mathrm{folk}}\subseteq \W_{\omega}^{\mathrm{Th}}$. In particular, the identity functor $\omega\Cat \to \omega\Cat$ induces a functor
        \[
        \Ho^{\mathrm{folk}}(\omega\Cat) \to \Ho^{\mathrm{Th}}(\omega\Cat),
        \]
        and for every functor $d : A \to \omega\Cat$, we have a commutative triangle in $\Ho^{\mathrm{Th}}(\omega\Cat)$:
        \[
        \begin{tikzcd}[column sep=tiny]
          \displaystyle\hocolim_A^{\omega\Cat,\mathrm{folk}}(d) \ar[dr] \ar[rr]& & \displaystyle\hocolim_A^{\omega\Cat,\mathrm{Th}}(d) \ar[dl]\\
          &\displaystyle\colim_A^{\omega\Cat}(d).
        \end{tikzcd}
        \]
        In the case $d$ is the functor $a \mapsto A/a$, we have already proved that the two slanted arrows of the previous triangle were isomorphisms. Hence, the canonical map
        \[
        \hocolim_{a \in A}^{\omega\Cat,\mathrm{folk}}A/a \to \hocolim_{a \in A}^{\omega\Cat,\mathrm{Th}}A/a
        \]
        is an isomorphism. Consider now the commutative triangle induced by the nerve functor:
        \[
        \begin{tikzcd}[column sep=tiny]
          &\displaystyle\hocolim_A^{\Psh{\Delta}}(N_{\omega}(d))\ar[dl] \ar[dr]&\\
          \displaystyle N_{\omega}(\hocolim_A^{\omega\Cat,\mathrm{folk}}(d)) \ar[rr]&&\displaystyle N_{\omega}(\hocolim_A^{\omega\Cat,\mathrm{Th}}(d)).
        \end{tikzcd}
        \]
        We have already seen that the slanted arrow on the right is an isomorphism. In the case that $d$ is the functor $a \mapsto A/a$, it follows from what we have proved that the horizontal arrow is an isomorphism. Hence, the canonical map
        \[
      \hocolim_{a \in A}^{\Psh{\Delta}}N_{\omega}(A/a) \to  N_{\omega}(\hocolim_{a \in A}^{\omega\Cat,\mathrm{folk}}A/a)
      \]
      is an isomorphism. \qedhere
      \end{description}
    \end{proof}
    We can now prove the main theorem of this paper.
    \begin{theorem}\label{maintheorem}
      Let $A$ be a small category seen as an object of $\omega\Cat$. The canonical comparison map
      \[
      \kappa N_{\omega}(A) \to \sH^{\mathrm{pol}}(A)
      \]
      is an isomorphism of $\Ho(\Ch)$.
    \end{theorem}
    \begin{proof}
      We have a commutative diagram
      \[
      \begin{tikzcd}[column sep=small]
       \displaystyle \hocolim_{a \in A}\kappa N_{\omega}(A/a) \ar[d]\ar[r]  &\displaystyle \kappa N_{\omega}(\hocolim_{a \in A}A/a) \ar[r] \ar[d] & \displaystyle\kappa N_{\omega}(\colim_{a \in A}A/a) \ar[r] \ar[d]& \kappa N_{\omega}(A)\ar[d]\\
       \displaystyle\hocolim_{a \in A}\sH^{\mathrm{pol}}(A/a) \ar[r] & \displaystyle\sH^{\mathrm{pol}}(\hocolim_{a \in A}A/a) \ar[r]& \displaystyle\sH^{\mathrm{pol}}(\colim_{a \in A}A/a) \ar[r] & \sH^{\mathrm{pol}}(A),
       \ar[phantom,from=1-1,to=2-2,"(A)",description]\ar[phantom,from=1-2,to=2-3,"(B)",description]\ar[phantom,from=1-3,to=2-4,"(C)",description]
      \end{tikzcd}
      \]
      where the vertical arrows are induced by the canonical comparison map. The goal is to show that the right vertical map of square $(C)$ is an isomorphism.

      By Lemma \ref{lemma:colimslice} the horizontal arrows of square $(C)$ are
      isomorphisms. By Theorem \ref{thm:homcolfolk}, the horizontal arrows of
      square $(B)$ are isomorphisms. By Lemma \ref{lemma:homcoltho}, Proposition
      \ref{prop:leftquilcocontinuous} and Lemma \ref{lemma:normalizedquillen}, the top
      horizontal arrow of square $(A)$ is an isomorphism, and by Proposition
      \ref{prop:leftquilcocontinuous} the bottom horizontal arrow of the same
      square is an isomorphism. Finally, from Proposition
      \ref{prop:comparisoncontractible}, Lemma \ref{lemma:terminalobject} and
      the fact that for every object $a$ of $A$ the category $A/a$ has a
      terminal object, we deduce that the left vertical arrow of square $(A)$ is
      an isomorphism. By 2-out-of-3 property for isomorphisms, this shows what
      we wanted.
    \end{proof}
    \section{Complement: a folk Theorem A}
    \begin{paragr}
      As a corollary of Theorem \ref{thm:homcolfolk}, we obtain Proposition
      \ref{prop:folkthmA} below\footnote{Since the time of writing of this
        article, the author has proved a generalization of this result, which
        can be found in \cite{guettaPHD}. However, the proof of this new result is
        completely different from the one used here and both are interesting in
        their own right.}. It is to be compare with Theorem $A$ of Quillen
      \cite[Theorem A]{quillen1973higher} and its generalization for
      $\omega$\nbd{}categories by Ara and Maltsiniotis \cite{ara2018theoreme}
      and \cite{ara2020theoreme}. However, note that in the pre-cited
      references, the weak equivalences considered are the ones induced by the
      nerve functor $N_{\omega}$ (which we called \emph{Thomason equivalences}
      in the proof of Lemma \ref{lemma:homcoltho}), whereas in the proposition below we work with the weak equivalences of the Folk model structure (\ref{paragr:defeqomegacat}).
    \end{paragr}
    \begin{proposition}\label{prop:folkthmA}
      Let
      \[
      \begin{tikzcd}
        X \ar[rr,"u"] \ar[dr,"v"'] &&Y \ar[dl,"w"] \\
        &A&
      \end{tikzcd}
      \]
      be a commutative triangle in $\omega\Cat$ and suppose that $A$ is a $1$\nbd{}category. If for every $a \in \Ob(A)$, the induced arrow
      \[
      u/a : X/a \to Y/a
      \]
      is an equivalence of $\omega$\nbd{}categories, then $u$ is also an equivalence of $\omega$\nbd{}categories.
    \end{proposition}
    \begin{proof}
      Consider the commutative square in $\Ho(\omega\Cat)$:
      \[
      \begin{tikzcd}
        \displaystyle\hocolim_{a \in A}X/a \ar[d] \ar[r] & \displaystyle\colim_{a \in A}X/a \ar[d] \\
                \displaystyle\hocolim_{a \in A}Y/a  \ar[r] & \displaystyle\colim_{a \in A}Y/a 
      \end{tikzcd}
    \]
    where the vertical arrows are induced by the arrows
    \[
    u/a : X/a \to Y/a.
    \]
    Since we supposed that these arrows were weak equivalences of the Folk model structure, it follows that the left vertical arrow of the previous square is an isomorphism. From Theorem \ref{thm:homcolfolk}, the horizontal arrows are isomorphisms.
    
    This proves that the right vertical arrow is also an isomorphism but it follows from \ref{paragr:unfoldfunctor} and Lemma \ref{lemma:colimslice} that this arrow can be identified with the image of $u : X \to Y$ in $\Ho(\omega\Cat)$. 
    \end{proof}
\section*{Acknowledgment}
I am infinitely grateful to Georges Maltsiniotis for explaining homotopy colimits to me. His point of view on the subject has infused into me, and into this article.

    \bibliographystyle{alpha}
    \bibliography{hompolcat}
\end{document}